\documentclass[10pt,reqno,oneside,dvipsnames]{amsart}

\usepackage{amsmath,amsthm, amssymb,eufrak,yfonts}
\usepackage{geometry} 
\usepackage[english]{babel} 
\usepackage{nicefrac} 
\usepackage{hyperref}
\usepackage{graphicx}
\usepackage{color}
\usepackage[usenames,dvipsnames]{xcolor}
\usepackage{dsfont}
\usepackage[sc]{mathpazo}
\usepackage{subcaption}
\usepackage{pgfplots}
\numberwithin{equation}{section}

\newtheorem{proposition}{Proposition}[section]
\newtheorem{theorem}[proposition]{Theorem}

\newtheorem{lemma}[proposition]{Lemma}

\theoremstyle{definition}
\newtheorem{dfz}[proposition]{Definition}
\newtheorem{rmk}[proposition]{Remark}

\newcommand{\beq}{\begin{equation}}
\newcommand{\eeq}{\end{equation}}
\newcommand{\ben}{\begin{enumerate}}
\newcommand{\een}{\end{enumerate}}
\newcommand{\bit}{\begin{itemize}}
\newcommand{\eit}{\end{itemize}}

\title[Optimization of the survival threshold for logistic equations]{Optimization of the survival threshold for anisotropic logistic equations with mixed boundary conditions}
\date{}
\author{Serena Benigno}
\address[Serena Benigno]{Università degli studi della Campania Luigi Vanvitelli}
\email{serena.benigno@unicampania.it}

\DeclareMathOperator{\tr}{Tr}
\begin{document}
\numberwithin{equation}{section}
\begin{abstract}
In this paper we study a reaction diffusion 
problem with anisotropic diffusion  and
mixed  Dirichlet-Neumann boundary conditions on 
$\partial \Omega$ for $\Omega\subset \mathbb{R}^N$, $N\geq 1$. 
First, we prove that the parabolic problem has a unique positive, bounded solution.
Then, we show that this solution converges as $t\to +\infty$ to the unique nonnegative solution of the elliptic associated problem. The existence of 
the unique positive solution to this problem depends on a
 principal eigenvalue of a  suitable  linearized problem with a sign-changing weights. Next, we 
study the  minimization of such eigenvalue with respect to 
the sign-changing weight, showing that there exists an optimal
bang-bang weight, namely a piece-wise constant weight that takes only two values. Finally, we completely solve the 
problem in dimension one. 
\end{abstract}
\maketitle
\noindent
{\footnotesize \textbf{AMS-Subject Classification}}. 
{\footnotesize 35J57, 35K61, 49K20}

\newcommand{\Bcal}{{\mathcal{B}}}
\newcommand{\Ccal}{{\mathcal{C}}}
\newcommand{\Dcal}{{\mathcal{D}}}
\newcommand{\Ecal}{{\mathcal{E}}}
\newcommand{\Fcal}{{\mathcal{F}}}
\newcommand{\Gcal}{{\mathcal{G}}}
\newcommand{\Hcal}{{\mathcal{H}}}
\newcommand{\Lcal}{{\mathcal{L}}}
\newcommand{\Mcal}{{\mathcal{M}}}
\newcommand{\Ncal}{{\mathcal{N}}}
\newcommand{\Pcal}{{\mathcal{P}}}
\newcommand{\Ocal}{{\mathcal{O}}}
\newcommand{\Qcal}{{\mathcal{Q}}}
\newcommand{\Rcal}{{\mathcal{R}}}
\newcommand{\Scal}{{\mathcal{S}}}
\newcommand{\Tcal}{{\mathcal{T}}}
\newcommand{\Ucal}{{\mathcal{U}}}
\newcommand{\Zcal}{{\mathcal{Z}}}
\newcommand{\Acal}{{\mathcal{A}}}
\newcommand{\Kcal}{{\mathcal{K}}}
\newcommand{\Jcal}{{\mathcal{J}}}
\newcommand{\Ical}{{\mathcal{I}}}
\section{Introduction}
This paper is devoted to
the study of the following reaction diffusion problem 
\begin{equation}\label{problema parabolico}
\begin{cases}
\partial_t v -d \ {\rm div}(\Hcal(\nabla v)
\nabla_\xi \Hcal(\nabla v))= f(x,v) &\text{in 
} \Omega\times (0,T), 
\\
\hskip127pt v=0  &\text{on }   
\Gamma_\Dcal\times (0,T), \\
\hskip42pt \Hcal(\nabla v)\nabla_\xi 
\Hcal(\nabla v)\cdot n=0   &\text{on}  \ 
\Gamma_\Ncal \times (0,T),\\
\hskip103pt v(x,0)=v_0(x)\  &\text{in }  
\Omega,
\end{cases}
\end{equation}
 settled in a bounded domain 
$\Omega\subset \mathbb{R}^N$ of class $C^2$
($N\geq 1$).
For $N\geq 2$, $\Gamma_\Dcal$, $
\Gamma_{\Ncal}$
are smooth $(N-1)$-dimensional submanifolds 
of $\partial\Omega$ such that $|\Gamma_\Dcal|
>0$, $\Gamma_\Dcal \cap 
\Gamma_\Ncal=\emptyset$, $
\overline{\Gamma_\Dcal} \cup 
\overline{\Gamma_\Ncal}=\partial\Omega$, $
\overline{\Gamma_\Dcal} \cap 
\overline{\Gamma_\Ncal}=\Gamma$ is a smooth $
(N-2)$-dimensional submanifold, and $n$ is 
the outer normal on $\partial\Omega$. 
For $N=1$, $\Gamma_\Dcal$ and $\Gamma_\Ncal$
are just the extrema of the interval $\Omega$.
\\
This class of problems usually arises in the description of 
the  dispersal of a population in a heterogeneous 
environment,
with mixed boundary conditions.
The function $v(x,t)$ represents the
population density in position $x$ at time
$t$, $d$ is the diffusion coefficient (the
motility of the population).
The initial datum $v_0$ is positive 
and 
it lies in
$L^\infty(\Omega)\cap H_\Dcal^1(\Omega)$,
where $H_\Dcal^1(\Omega)$ is the subspace of
$H^1(\Omega)$ containing the functions with
zero trace on $\Gamma_\Dcal$ (see Definition
\eqref{spazio}).
The diffusion operator is induced by an anisotropic norm, given by
 the function $\Hcal:
\mathbb{R}^N \rightarrow \mathbb{R}$, 
belonging to $C^2(\mathbb{R}^N \setminus 
\left\{ 0 \right\})$, and such that it satisfies the following 
properties
\begin{equation}\label{H positiva}
\Hcal(\xi)\geq 0 \ \text{and} \ \Hcal(\xi)=0 
\ \text{if and only if} \ \xi=0;
\end{equation}
\begin{equation}\label{H pos omogenea}
\Hcal(t \xi)=t \Hcal(\xi) \ \text{for all} \ 
t\geq 0, \ \text{for all }\xi \in \mathbb{R}^N;
\end{equation}
\begin{equation}\label{convessità}
\left\{ \xi \ : \ \Hcal(\xi)<1 \right\} \ \text{is uniformly convex};
\end{equation}
where assumption \eqref{convessità}
means that the mean curvatures of the sub-level  
set $\left\{ \xi \ : \ \Hcal(\xi)<1 \right\} $ 
are positive and bounded away from zero.\\
If $\Hcal(\xi)=|\xi|$ and $f(x,s)= 
m(x)s-s^2$, equation in 
\eqref{problema parabolico}
becomes a largely studied reaction diffusion
equation introduced independentely by
Fisher \cite{13} and 
Kolmogorov, Petrovskii and 
Piskunov \cite{17}.
In this classical context, 
under homogeneous Dirichlet boundary conditions,
Skellam \cite{23} proved that, in the case of 
$m$ constant and $N=2$,
for every $v_0\geq 0$
there exists $v$ globally defined and
$v(x,t)\to u(x)$ as $t\to +\infty$, where $u$ is 
the unique positive solution to
\begin{equation*}
\begin{cases}
-\Delta u= m(x)u-u^2 \ &\text{in } \Omega,\\
\hskip6pt
\hskip10pt u=0 &\text{on } \partial\Omega.
\end{cases}
\end{equation*}
In their seminal papers \cite{8,10}, Cantrell and Cosner
proved the same result for
$m$ nonconstant and for every dimension. 
Here, we assume that the nonlinearity
$f(x,s):\Omega 
\times [0,+\infty)\rightarrow \mathbb{R}$
is a Carathéodory function satisfying the 
following conditions
\begin{equation}\label{f i}
f(x,0)=0;
\end{equation}
\begin{equation}\label{f ii} 
\exists \ M\geq 0 \ : f(x,s)<0,
\ \forall \, s\geq M, \ \text{a.e. in }
\Omega;
\end{equation}
\begin{equation}\label{f iii}  \frac{f(x,s)}
{s} \ \text{is strictly decreasing for} \ 
s>0.
\end{equation}
Moreover, we suppose that there exists $\partial_s f(x,s)$, partial derivative of
$f$ with respect to $s$,   for almost every $x\in \Omega$. In addition, $\partial_s f(x,s)$ satisfies the  following growth condition. 
\begin{equation} \label{f iv} 
|\partial_s f(x,s)|\leq L \ \text{for 
a.e. } x\in \Omega \ \text{and for all } |s|\leq A, \, A>0.
\end{equation}
Since we will be focused on positive 
solutions, we consider $f \equiv 0$ on
$\Omega\times[-\infty,0)$.
Berestycki, Hamel and Roques proved in 
\cite{Berestychi} that when homogenenous
Dirichlet boundary conditions are imposed,
there exists a unique $v$ solution of the parabolic problem, $v$ is globally
defined and $v(x,t)\to u(x)$ as $t \to \infty$ if and only 
if $\mu_1(d,m)<0$, where $\mu_1(d,m)$ is given 
by
\begin{equation*}
\mu_1(d,m)=\inf \left\{ \frac{d\int_\Omega 
|\nabla u|^2-\int_\Omega mu^2}
{\int_\Omega u^2} \ : \ u \in H_	
0^1(\Omega) \ , \ u\geq 0 \ , u 
\not\equiv 
0  \right\},
\end{equation*}
and $m$ is the right derivative of $f$ in $0$ 
with respect to $s$, namely
\begin{equation}\label{limite f}
\lim_{s\to0^+}\dfrac{f(x,s)}s:=m(x) \in 
L^\infty(\Omega).
\end{equation}
Denoting with $d^*=\sup\left\{ d>0 \ : \ 
\mu_1(d,m)<0 \right\}$, one can show that $d^*
\in (0,+\infty)$ and the number $\lambda_1(m)=
1/d^*$ is a principal positive
eigenvalue (i.e. a positive
eigenvalue with associated positive eigenfunction) of 
\begin{equation*}
\begin{cases}
-\Delta u = \lambda_1(m) m u \ &\text{in } \Omega,\\
\hskip15pt u>0 &\text{in } \Omega,\\
\hskip15pt u=0 &\text{on } \partial\Omega.
\end{cases}
\end{equation*}
Cantrell and Cosner
\cite{8,10} showed that $\lambda_1(m)$ is 
achieved
for $m$ varying in 
the class $\Mcal$ defined as
\begin{equation}\label{classe M}
\Mcal:=\left\{ m \in L^\infty(\Omega) \ :  \ 
m^+\not\equiv 0 \ , \ -\beta \leq m \leq 1 \ 
, \ \int_\Omega m \leq m_0|\Omega| 
\right\},
\end{equation}
with $\beta \in \mathbb{R}^+$
and $m_{0}\in \mathbb{R}$
such that $m_0\in (-1,\beta)$. Moreover, 
their results are the first contributions in 
the maximization of $d^*$ with respect to $m$,
that is equivalent to solving the 
minimization problem
\begin{equation}\label{prob lap}
\Lambda_1=
\inf\left\{ \lambda_1(m) \ : \ m \in \Mcal 
\right\}.
\end{equation}
This optimization problem has been studied also
for Neumann or even 
Robin boundary conditions (see for instance
\cite{HKL,18,21}). In all these papers it is
shown that $\Lambda_1$ is achieved by a 
bang-bang
weight, namely $m=\chi_E-\beta \chi_{E^c}$, 
where $E$ is a superlevel set of the associated
positive
eigenfunction. Moreover, in dimension 1 it is
known that $E$ is an interval centered at the
point of maximal distance from the boundary
when homogeneous Dirichlet boundary conditions
are imposed (see e.g. \cite{8,10}), while
it is attached to the boundary
under homogeneous Neumann boundary conditions
(see e.g. \cite{21}). The case of Robin boundary
conditions is somewhat intermediate and it is
analyzed in \cite{HKL,21}.
The optimization of $\lambda_1(m)$ has been 
recently studied under Robin boundary conditions
for a function 
$\Hcal \in C^2(\mathbb{R}^N \setminus
\left\{ 0 \right\})$ satisfying \eqref{H positiva}-\eqref{convessità} in \cite{Pepisc},
where it is shown that the anisotropic 
counterpart of $\lambda_1(m)$ is still reached 
by
a bang-bang weight associated with an optimal 
set
that, in dimension 1, is again an interval 
whose position, in the case of homogeneous 
Dirichlet and Neumann boundary conditions is determined by $\Hcal$,
differently from the Laplacian
case.\\
As relevant is the case of mixed boundary
conditions, where a part of the buondary of
$\Omega$ with positive measure ($\Gamma_\Dcal$)
is supposed to be completely hostile, while $
\Gamma_\Ncal$ represents the part where there is
no flux.\\
Up to our knowledge, this context has not been
studied yet, even for the Lapalacian operator and we
address it here in the presence of an
anisotropic diffusion, aiming at detecting a
possible relationship between the boundary
conditions and the diffusion. With this goal
in mind, we first prove the existence of a 
survival threshold also in this context.
Our first result is concerned with the elliptic
problem associated with \eqref{problema 
parabolico}.
\begin{theorem}\label{esistenza sotto d star}
Let
$\Hcal \in C^2(\mathbb{R}^N\setminus 
\left\{ 0 \right\})$ satisfy
\eqref{H positiva}-\eqref{convessità}, and
let $f$ satisfy 
assumptions \eqref{f i}-\eqref{f iv}.
If $m$ defined in \eqref{limite f} is such that
$m^+\not\equiv 0$,
then there exists $d^*=d^*(m)>0$ such that
\begin{itemize}
\item[(i)] if $0<d<d^*$,
then there
exists a unique, positive, bounded weak 
solution 
to
problem 
\begin{equation}\label{problema non lineare}
\begin{cases}
-d \ \text{div}(\Hcal(\nabla u)\nabla_\xi \Hcal 
(\nabla u))= f(x,u) \ &\text{in} \ \Omega, \\
\hskip110pt u>0 &\text{in } \Omega,\\
\hskip110pt u=0 &\text{on } \Gamma_\Dcal,\\ 
\hskip34pt \Hcal(\nabla u)\Hcal_\xi(\nabla u)
\cdot n=0  &\text{on } \Gamma_\Ncal.
\end{cases}
\end{equation}
Moreover, $u \in C^{0,\gamma}
\left(\overline{\Omega}\right)$, for some $
\gamma \in (0,1/2)$;
\item[(ii)] if $d\geq d^*$, then the only
nonnegative weak solution to Problem
\eqref{problema non lineare} is the trivial 
one.
\end{itemize}
\end{theorem}
Let us we observe that Theorem \ref{esistenza sotto d 
star}
is original even in the case of $\Hcal(\xi)=
|\xi|$ with mixed boundary conditions.
Moreover, it holds with slight modifications also for homogeneous
Dirichlet and Neumann boundary conditions,
so that it generalizes the existence results contained in
\cite{Berestychi, 21, Pepisc} with respect  both  to the boundary conditions
and to  anisotropic diffusion.

We observe that, as stated in Theorem 
\ref{esistenza sotto d star}, the nontrivial
solutions to
Problem \eqref{problema non lineare} have an 
upper bound in terms of regularity in 
$\overline{\Omega}$
, due to the
lack of regularity on $\Gamma$.
In 
particular, Shamir \cite{Shamir} proved 
that the highest regularity one can
obtain is H\"older regularity with exponent 
$\gamma<1/2$.
Here, we can adapt the results of \cite{Stampacchia} (see also \cite{Colorado Peral, Colorado Peral 2004}) to obtain the regularity property of the solution in this anisotropic setting.
\\
The next result shows how the threshold
$d^*$ governs extinction and survival.
\begin{theorem}\label{thm par}
Let 
$v_0 \in L^\infty(\Omega)\cap H_\Dcal^1(
\Omega)$ be a positive function in $\Omega$.
Let
$\Hcal \in C^2(\mathbb{R}^N\setminus 
\left\{ 0 \right\})$ satisfy
\eqref{H positiva}-\eqref{convessità}, and
let $f$ satisfy 
assumptions \eqref{f i}-\eqref{f iv}.
If If $m$ defined in \eqref{limite f} is such that
$m^+\not\equiv 0$,
then the following statements hold.
\begin{itemize}
\item[(1)]
For every $0<d<d^{*}$, there exists 
$v\in L^{2}(0,T; H^{1}_{\Dcal}(\Omega))\cap 
C([0,T]; L^{2}(\Omega))$, with
$\partial_t v \in 
L^2(\Omega\times(0,T))$, unique, positive, bounded  
weak solution to Problem
\eqref{problema parabolico}.
Moreover, 
\[
\lim_{t\to+\infty} v(x,t)=u(x),
\]
where $u$ is the unique, positive, bounded weak solution
to Problem \eqref{problema non lineare}.
\item[(2)] For every $d \geq d^*$ there 
exists
$v\in L^{2}(0,T; H^{1}_{\Dcal}(\Omega))\cap 
C([0,T]; L^{2}(\Omega))$ unique, nonnegative, bounded 
weak solution to Problem \eqref{problema 
parabolico},
and it results
$$\lim_{t\rightarrow +\infty}v(x,t)=0.$$
\end{itemize}
\end{theorem}
Theorems \ref{esistenza sotto d star} and
\ref{thm par} will be proved adopting the
following strategy. First of all, we show that 
the existence of a positive weak solution
to Problem \eqref{problema non lineare}
depends on the sign of the eigenvalue
$\mu(d,m)$,
given by
\begin{equation}\label{def mu}
\mu(d,m):=\inf \left\{ \frac{d\int_\Omega 
\Hcal^2(\nabla u)-\int_\Omega mu^2}
{\int_\Omega u^2} \ : \ u \in H_	
\Dcal^1(\Omega) \ , \ u\geq 0 \ , u 
\not\equiv 
0  \right\},
\end{equation}
for every $d>0$ and for every $m \in L^\infty
(\Omega)$ (see Theorem \ref{esistenza 
ellittico}). 
Then, the same sign assumption on $\mu(d,m)$
yields the existence result for the parabolic
Problem \eqref{problema parabolico}
(see Theorem \ref{es par mu}).
Subsequently, in Section \ref{lambda}
we show that the principal eigenvalue
$\lambda(m)$
of
problem
\begin{equation}\label{problema anisotropo}
\begin{cases}
- \text{div}(\Hcal(\nabla u)\nabla _\xi 
\Hcal(\nabla u))=\lambda mu  & \text{in} \ 
\Omega ,\\
\hskip102pt \; u >0 & \text{in} \ \Omega,\\
\hskip102pt \; u=0 & \text{on } \ \Gamma_\Dcal,
\\ 
\hskip19pt \Hcal(\nabla u)\nabla_\xi 
\Hcal(\nabla u)
\cdot n=0 & \text{on } \ \Gamma_\Ncal,
\end{cases}
\end{equation}
defined, for every $m\in \Mcal$, as
\begin{equation}\label{def lambda}
\lambda(m):=\inf\left\{ \frac{\int_\Omega 
\Hcal^2(\nabla u) }{\int_\Omega mu^2} \ : 
\ u \in H_\Dcal^1(\Omega) \ , \ u\geq 0 \ , \ 
\int_\Omega mu^2>0 \right\},
\end{equation}
is achieved (see Proposition \ref{lambda 
minimo}).
This allows us to introduce the survival 
threshold and complete the proofs of Theorems
\ref{esistenza sotto d star} and \ref{thm par}.
\\
As for the Laplacian operator
with homogeneous
Dirichlet, Neumann or Robin boundary conditions,
it turns out that $d^*=1/\lambda(m)$,
thus, to maximize  the ranges of diffusion coefficients for which we have a positive asymptotical state, we need to 
study the minimization problem
\begin{equation}\label{problema di minimizzazione spettrale}
\Lambda:=\inf\left\{ \lambda(m) \ : \ 
m \in \Mcal \right\}.
\end{equation}
In the framework of mixed boundary conditions, 
Denzler \cite{Denzler,Denzler'}
studied the
behavior of the first eigenvalue of the 
Laplacian
operator with constant weight, 
depending on  
the configuration of $\Gamma_\Dcal$. He showed 
that
the minimum of
such eigenvalue on a class of subsets of $
\partial
\Omega$ with a measure constraint is taken on 
(see \cite[Theorem 5]{Denzler}). Moreover, if 
$\Omega$ is a ball, the Dirichlet 
part
of the boundary realizing the minimum is a 
spherical
cap (see \cite[Theorem 3]{Denzler'}). Later, in 
\cite{LMPPS} 
Leonori, Medina, Peral, Primo and
Soria extended
Denzler's results to the case of a nonlocal operator.
They chacarterized the arrangement of a sequence of
Neumann boundary parts, in order to prove that the sequence of the
corresponding first eigenvalues gets close to the first
eigenvalue of the same operator with 
homogeneous Dirichlet 
boundary conditions (see \cite[Theorem 1.4]{
LMPPS}).\\
In this work we have fixed $\Gamma_\Dcal$ and
$\Gamma_\Ncal$, and we want to find the
weight representing the optimal configuration
of resources inside the domain, i.e. the 
one realizing the minimal principal eigenvalue
\eqref{problema di 
minimizzazione spettrale}.
In analogy to what has been done in 
\cite{Pepisc}, 
here we perform the analysis under mixed
boundary conditions; first of all we prove
that there exists an optimal weight
which is of bang-bang type. Namely, the 
following result holds.
\begin{theorem}\label{minimo minimo}
Let $\Hcal \in C^2(\mathbb{R}^N 
\setminus \left\{ 0 \right\})$, $N\geq 1$, 
and assume \eqref{H 
positiva}-\eqref{convessità}. Problem 
\eqref{problema di minimizzazione spettrale}
is solved by $m_\omega(x)=\chi_\omega-
\beta\chi_{\omega^c}$, 
where $\omega$ is a measurable subset of $
\Omega$ such that $|\omega|
=\frac{\beta+m_0}{1+\beta}|\Omega|$. 
Moreover, if $
\varphi=\varphi(m_\omega)$ is the positive, 
bounded eigenfunction associated with $
\Lambda=\lambda(m_\omega)$, then $\omega$ is 
a superlevel set of 
$\varphi$, i.e. 
$$
\omega=\left\{ \varphi>t \right\}
$$  
for some $t>0$,
and every level set of $\varphi$ 
has zero measure.
\end{theorem}
This result shows that $
\Lambda=\lambda(m_{\omega})$, solution to
Problem
\eqref{problema di minimizzazione 
spettrale}, is equivalently a solution
to the problem
\[
\min\left\{\lambda(\chi_{E}-
\beta\chi_{E^{c}}), \; E\subset \Omega, \ 0<|
E|\leq\delta\right\}
\]
for every $\delta \in \left(0, \frac{\beta+
m_{0}}{1+\beta}|\Omega|\right]$.
From now on,  we use the 
notation $
\Lambda=\lambda(m_\omega)=\lambda(\omega)$.
\\
Then, we pursue the study in dimension one,
where
assumptions \eqref{H positiva}, \eqref{H 
pos omogenea} and \eqref{convessità} imply 
that $\Hcal$ is given by
\begin{equation}\label{H unidimensionale}
\Hcal(x)=
\begin{cases}
\hskip8pt ax \ &\text{if} \ x>0, \\
-bx  &\text{if} \ x\leq 0,
\end{cases}
\end{equation}
with $a \neq b$, and, taking $
\Omega=(0,1)$, Problem \eqref{problema 
anisotropo} becomes 
\begin{equation}\label{problema unidmensionale DN}
\begin{cases}
\hskip7pt -(\Hcal(u')\Hcal'(u'))'=\lambda 
m u \  
&\text{in } (0,1), \\
\hskip77pt u>0 &\text{in } (0,1), \\
\hskip63pt u(0)=0, \\
\Hcal(u'(1))\Hcal'(u'(1))=0, \\
\end{cases}
\end{equation}
or 
\begin{equation}\label{problema unidimensionale ND}
\begin{cases}
\hskip7pt -(\Hcal(u')\Hcal'(u'))'=\lambda 
m u  \ 
&\text{in } (0,1), \\
\hskip77pt u>0 &\text{in} \ (0,1),\\
\hskip63pt u(1)=0, \\
\Hcal(u'(0))\Hcal'(u'(0))=0.
\end{cases}
\end{equation}
In this context we can characterize the optimal 
set $\omega$ as shown in the next 
result.
\begin{theorem}\label{D intervallo}
Let $N=1$, and $\Hcal$ be as in \eqref{H 
unidimensionale}. Then, the minimal 
principal eigenvalue associated with Problem
\eqref{problema unidmensionale DN} is 
$\Lambda=\lambda(\omega_1)$, where $\omega_1=\left( \frac{1-m_0}{1+\beta} , 
1\right)$, and the minimal principal eigenvalue
associated with problem \eqref{problema unidimensionale ND} is $\Lambda=\lambda(\omega_2)$, where $\omega_2=\left( 0,\frac{\beta+m_0}{
1+\beta}\right)$.
\end{theorem}
This Theorem shows that, at least in dimension 
one, the anisotropy does not control where 
the optimal set is located,  this fact 
only depends on the boundary conditions.
Let us observe that Theorem \ref{D intervallo} 
is new also in the case 
$\Hcal(\xi)=|\xi|$ and it enlightens the
strong dependence of the optimal set on
the boundary conditions.
\\
The paper is organized as follows.
In section \ref{setting} we give the general 
setting of the problem, then in Section 
\ref{parabolico} we study the parabolic
Problem \eqref{problema parabolico} and
the elliptic Problem
\eqref{problema non lineare}. Section
\ref{lambda} is devoted to the study of the
principal eigenvalue $\lambda(m)$ and to the 
proof of Theorems \ref{esistenza sotto d star},
\ref{thm par},
\ref{minimo minimo} and \ref{D intervallo}.
Finally, in appendix \ref{reg} 
we study the regularity properties of a weak 
solution to Problems \eqref{problema non 
lineare}.
\section{Setting of the Problem}\label{setting}
As preliminar observations regarding the 
function $\Hcal$, let us note that,
under the assumptions \eqref{H positiva} and 
\eqref{H pos omogenea}, $\Hcal $ satisfies
the following growth condition
\begin{equation}\label{condizione di crescita H}
\underline{\alpha}|\xi|\leq \Hcal(\xi)\leq 
\overline{\alpha}|\xi| \ \ , \ \text{with} \ 
\overline{\alpha}>\underline{\alpha}>0 \ , \ 
\text{for all } \xi \in \mathbb{R}^N\setminus 
\left\{ 0 \right\}.
\end{equation}
In view of  \eqref{convessità}, one can prove 
that $\Hcal^2$ is strictly convex, and that 
there exists $\Xi>0$ such that for every $\xi 
\in \mathbb{R}^N\setminus \left\{ 0 \right\}$ 
one has
\begin{equation}\label{hessiano H}
\text{Hess}(H^2)(\xi)_{ij}\zeta_i 
\zeta_j \geq \Xi |\zeta|^2.
\end{equation}
Moreover, \eqref{H pos omogenea} yields that 
\begin{equation}\label{(2.5) Montoro}
\nabla_\xi \Hcal(\xi)\cdot\xi 
=\Hcal(\xi) \ \ \text{for all} \ \xi 
\in \mathbb{R}
^N \setminus \left\{ 0 \right\},
\end{equation}
and again assumption \eqref{H pos omogenea} 
implies the 0-homogeneity of $\nabla 
\Hcal$, thus there exists
$Q>0$ such that
\begin{equation}\label{(2.6) Montoro}
|\nabla_\xi \Hcal(\xi)|\leq Q , \ \ \text{for 
all} \ \xi \in \mathbb{R}^N \setminus 
\left\{ 0 \right\}.
\end{equation}
Note that
assumption \eqref{f iii} implies that
there exists $\lim_{s\rightarrow 0^+}
\frac{f(x,s)}{s}$. In condition \eqref{limite 
f} it is assumed $m\in L^\infty(\Omega)$.
Moreover, 
hypotheses \eqref{f iii} and \eqref{limite f}
yield that
\begin{equation}\label{rmk coercività}
f(x,s)< m(x)s \quad \text{for all } s
>0.
\end{equation}
Indeed, for all $0<t<s$ one has
$$\frac{f(x,s)}{s}<
\frac{f(x,t)}{t},$$
then, taking the limit as $t \rightarrow 0^+$
one obtains \eqref{rmk coercività}.
We will study the problem in the functional 
space
\begin{equation}\label{spazio}
H_\Dcal^1(\Omega):=\left\{ u \in H^1(\Omega) 
\ : \ \tr (u)=0 \ \text{on} \ \Gamma_\Dcal 
\right\},
\end{equation}
where $ \tr (u)$ denotes the image of the 
trace operator acting on $u$.
According to \cite[Theorem 3.1]{4'}, 
$H_\Dcal^1(\Omega)$ is defined as the closure 
of 
the space 
$C_\Dcal^\infty(\Omega):=C_0^\infty\left(\Omega 
\cup \Gamma_\Ncal\right)$ 
with respect to the norm in $H^1(\Omega)$.
Since $H_\Dcal^1(\Omega)$  is a closed 
subspace  
of $H^1(\Omega)$,  it is  a separable Hilbert 
space. We will denote with
$\| \cdot \|$ the 
norm in $H_\Dcal^1(\Omega)$, with $\| \cdot\|
_p$ the norm in $L^p(\Omega)$, for every 
$1\leq p \leq \infty$, and with $(\cdot,
\cdot)_2$ the scalar product in 
$L^2(\Omega)$.\\
Let us observe that the norm in $H^1(\Omega)$ 
and the $L^2$-norm of the gradient are two 
equivalent norms in $H_\Dcal^1(\Omega)$, 
indeed, the following holds.
\begin{proposition}\label{poincare}
There exists $c>0$ such that $|| u ||_2 \leq 
c ||\nabla u||_2$, for all $u \in 
H_\Dcal^1(\Omega)$.
\end{proposition}
\begin{proof}
The proof can be obtained arguing as in 
\cite[Lemma 3.1]{12}, we include here 
details for the reader's convenience.
By contradiction, suppose that there 
exists $\left\{ u_k \right\}\subset 
H_\Dcal^1(\Omega)$ such that $$\int_\Omega |
u_k|^2 \geq k \int_\Omega |\nabla u_k|^2,$$ 
and up to a normalization we can take $\| 
u_k \|_2=1$. Then $\left\{ u_k \right\}$ is 
bounded in $H^1(\Omega)$, and since 
$H_\Dcal^1(\Omega)$ is weakly closed, there 
exists $u \in H_\Dcal^1(\Omega)$ such that, 
up 
to a sub-sequence, $$u_k \rightharpoonup u 
\ \text{in} \ H_\Dcal^1(\Omega).$$
Since $H^1(\Omega)$ is compactly embedded in 
$L^2(\Omega)$, one has
$$1=\lim_{k\rightarrow +\infty}\int_\Omega |
u_k|^2=\int_\Omega |u|^2.$$ 
Moreover, thanks to the lower semi-continuity 
of the $L^2$-norm of the gradient we know 
that $$\int_\Omega |\nabla u|^2 \leq 
\liminf_{k \rightarrow + \infty}\int_\Omega |
\nabla u_k|^2 \leq \liminf_{k \rightarrow + 
\infty} \frac{1}{k}\int_\Omega |u_k|^2=0,$$ 
thus $\nabla u=0$ a.e. in $\Omega$, and so 
$u$ is costant, but since $\tr(u)=0$ on $
\Gamma_\Dcal$, this means that $u\equiv 0$ 
a.e. 
in $\Omega$, which contradicts $\| u \|_2=1$.
\end{proof}
From now on we will consider 
$H_\Dcal^1(\Omega)$ endowed with the norm $\| 
u \|:=\| \nabla u \|_2$.\\
In particular, Proposition \ref{poincare} 
and assumption \eqref{condizione di crescita 
H} yield that $\| \Hcal(\nabla\cdot )\|_2$
is an equivalent norm in $H_\Dcal^1(\Omega)$.
\section{The nonlinear Problems}\label{parabolico}
This section is devoted to the study of the
parabolic and elliptic problems
\eqref{problema 
parabolico}
\eqref{problema non lineare}.
In particular, we will prove
Theorems \ref{esistenza sotto d star} and
\ref{thm par}.
Throughout this section we will always assume that
$\Hcal$ satisfies assumptions 
\eqref{H positiva}-\eqref{convessità}, and that
the nonlinearity $f$ satisfies 
\eqref{f i}-\eqref{f iv}.
We will be first focused on the elliptic Problem 
\eqref{problema non lineare},
obtaining the existence of a positive 
solution by showing the existence of an 
ordered couple of sub- and super-solutions 
defined as follows.
\begin{dfz}
We say that $u \in H^1(\Omega)$ is a weak 
super-solution to Problem \eqref{problema 
non lineare} if 
\begin{itemize}
\item[•]for all $\phi\in 
C_\Dcal^\infty(\Omega)
$, $\phi \geq 0$, one has $$d\int_\Omega 
\Hcal(\nabla u)\Hcal_\xi (\nabla u)
\cdot\nabla \phi-\int_\Omega f(x,u)\phi\geq 
0;$$
\item[•]$u\geq 0$ on $\Gamma_\Dcal$.
\end{itemize}
The definition of sub-solution can be given 
analogously, changing the sign of the 
inequalities.\\
$u \in H_\Dcal^1(\Omega)$ is a weak solution 
if it is both a sub- and super-solution.
\end{dfz}
\begin{proposition}\label{propo sopra e sotto soluzioni}
If there exist $\underline{u},\overline{u} 
\in H^1(\Omega)$ respectively  weak 
sub- and super-solutions to Problem 
\eqref{problema non lineare} such that $
\underline{c}\leq \underline{u}\leq 
\overline{u}\leq \overline{c}$, where $
\underline{c},\overline{c} \in \mathbb{R}$, 
then there exists $u \in H_\Dcal^1(\Omega)$ 
weak 
solution, such that 
$$\underline{u}\leq u \leq \overline{u} \ 
\text{a.e. in} \ \Omega.$$
\end{proposition}
\begin{proof} The proof closely follows 
\cite[Theorem 2.4]{24} (see also 
\cite{Pepisc}, where the authors prove it for 
Robin boundary conditions). Here we just 
point out the differences. \\
We set $$\Bcal:=\left\{ u \in 
H_\Dcal^1(\Omega) \ 
: \ \underline{u}\leq u \leq \overline{u} \ 
\text{a.e. in} \ \Omega \right\},$$
$\Fcal(x,u):=\int_0^{u(x)}f(x,s)ds$, and $
\Ecal:H_\Dcal^1(\Omega) \rightarrow 
\mathbb{R}$ as the functional 
$$\Ecal(u):=\frac{d}{2}\int_\Omega 
\Hcal^2(\nabla u)-\int_\Omega \Fcal(x,u).$$
Assumptions \eqref{f iv}, \eqref{condizione di crescita H} yield
that $\Ecal$ is coercive on $\Bcal$. 
Moreover,
the compact embedding of 
$H_\Dcal^1(\Omega)$ into $L^2(\Omega)$ 
implies that
$\Ecal$ is weakly 
lower semi-continuous.
Then, there exists
$u \in \Bcal$, minimum point of $\Ecal$.
Setting $$D(u,w):=d\int_\Omega 
\Hcal(\nabla u)\nabla_\xi \Hcal(\nabla u)
\cdot\nabla w- \int_\Omega f(x,u)w, \quad
\text{for all } w \in H_\Dcal^1(\Omega),$$ 
assumptions \eqref{condizione di crescita H} and \eqref{f iv}
allow us to apply Lebesgue Theorem and obtain
\begin{align*}
\langle d_G\Ecal(u),v-
u\rangle=\lim_{\epsilon \rightarrow 0}
\frac{\Ecal(u+\epsilon(v-u))-\Ecal(u)}
{\epsilon}=D(u,v-u)\geq 0.
\end{align*}
Where $d_G$ denotes the Gateaux derivative and the last inequality holds since $u$ is a minimum point on a convex set.
For every $
\varphi \in C_\Dcal^\infty(\Omega)$
we define $v_\epsilon:=u+\epsilon\varphi-
\varphi^\epsilon+\varphi_\epsilon$ and 
$$
\varphi^\epsilon:=\max \left\{ 0,u+
\epsilon\varphi-\overline{u} \right\} \ , \ 
\varphi_\epsilon:=\max\left\{ 0, 
\underline{u}-(u+\epsilon\varphi) \right\}.
$$
Let us show that $v_\epsilon\in H_\Dcal^1(\Omega)
$.
Since $\varphi \in C_\Dcal^\infty(\Omega)$ 
and $u \in H_\Dcal^1(\Omega)$, it results
\[
u+
\epsilon\varphi -\overline{u}= u -
\overline{u}\leq 0, \quad
\underline{u}-u-\epsilon\varphi=\underline{u}-u
\leq 0
\quad \text{on } \Gamma_\Dcal,
\]
so that 
$\varphi^\epsilon \equiv \varphi_\epsilon\equiv 0$ on 
$\Gamma_\Dcal$ (in the sense of the 
Trace operator), then $v_\epsilon\in 
H_\Dcal^1(\Omega)$.
From now on, the proof can be concluded as in 
\cite{Pepisc}.
\end{proof}
Propostion \ref{propo sopra e sotto 
soluzioni} guarantees that, to have 
existence of a bounded weak solution to 
Problem \eqref{problema non lineare}, it is 
sufficient to find $\underline{u}\leq 
\overline{u}$ bounded weak sub- and 
super-solutions respectively. In particular, 
we will see that 
the existence of a positive weak sub-solution
depends on the sign of
$\mu(d,m)$, defined in \eqref{def mu}.
Let us first prove that it is attained.
\begin{proposition}\label{l'autovalore è minimo}
For every $m \in L^\infty(\Omega)$, $\mu(d,m)$ 
defined in
\eqref{def mu} is attained by a function
$u \in H_\Dcal^1
(\Omega)$.
Moreover, $u \in 
C^{0,\gamma}
\left(\overline{\Omega}\right)$, for some $
\gamma 
\in (0,1/2)$.
\end{proposition}
\begin{proof}
Let us first point out that, 
in view of assumption \eqref{H pos omogenea}, 
 minimizing on the set of function $ v \in 
H_\Dcal^1(\Omega)$ and such that $v\geq 0 \ , \ 
 v\not\equiv 0 $ is equivalent 
to minimizing on $$\Gcal:=\left\{ v \in 
H_\Dcal^1(\Omega) \ : \  v \geq 0, \ 
\int_\Omega v^2 =1 \right\}.$$
Let us also observe that $\mu(d,m) \in 
\mathbb{R}$. Indeed, for every $u \in 
H_\Dcal^1(\Omega)$ with $\| u\|_2=1$, it
results
$$d\int_\Omega \Hcal^2(\nabla u)-
\int_\Omega mu^2\leq
d\overline{\alpha}^2\| \nabla u\|_2^2+
\| m \|_\infty.$$ 
Let  $\left\{ u_n \right\}
\subset \Gcal$ be a  minimizing sequence, then 
$\int_\Omega \Hcal^2(\nabla u_n)$ is bounded,
and thanks to assumption \eqref{condizione di 
crescita H} we get that 
$\left\{ u_n \right\}$
is bounded in $H_\Dcal^1(\Omega)$. As a consequence, 
there exists $u \in 
H_\Dcal^1(\Omega)$ such that, up to a sub-sequence, $u_n 
\rightharpoonup u$ in $H_\Dcal^1(\Omega)$, 
$u_n\rightarrow u$ strongly $L^2(\Omega)$, and 
almost everywhere. 
Hence $$1=\lim_{n\rightarrow +\infty} 
\int_\Omega 
u_n^2=\int_\Omega u^2,$$
so that $u \in \Gcal$, and $$d\int_\Omega 
\Hcal^2(\nabla u)- \int_\Omega mu^2\geq 
\mu(d,m).$$
Moreover, since $u_n \rightarrow u$ a.e. in 
$\Omega$, $u \geq 0$ in $\Omega$, and
\eqref{limite f} allows us to apply Lebesgue 
Theorem and obtain
$$\int_\Omega mu_n^2 \rightarrow 
\int_\Omega mu^2.$$ 
On the other hand,  the weak 
lower semicontinuity of the norm of $
\Hcal(\nabla \cdot)$ in $L^2$ yields
$$d\int_\Omega \Hcal^2(\nabla u)- 
\int_\Omega mu^2\leq \liminf_n \left( 
d\int_\Omega \Hcal^2(\nabla u_n)- 
\int_\Omega mu_n^2\right)= \mu(d,m),$$
so that $$d\int_\Omega \Hcal^2(\nabla u)- 
\int_\Omega mu^2=\mu(d,m).$$
This implies that $u\in H_\Dcal^1(\Omega)$ is a weak solution of
\begin{equation}\label{problema mu}
\begin{cases}
-d \text{div}(\Hcal(\nabla u)\nabla_\xi 
\Hcal(\nabla u))-mu=\mu(d,m)u \ &\text{in } 
\Omega,\\
\hskip135pt u\geq 0 &\text{in }  \Omega,\\
\hskip135pt u=0 &\text{on } \Gamma_\Dcal,\\
\hskip50pt \Hcal(\nabla u)\nabla_\xi 
\Hcal(\nabla u)\cdot n =0 &\text{on }  
\Gamma_\Ncal.
\end{cases}
\end{equation}
In order to prove that $u \in 
L^\infty(\Omega)$ one can argue as in
\cite[Theorem 3.1]{DG'}, here we only highlight 
the differences due to the
functional space $H_\Dcal^1(\Omega)$. Let $k$ be 
a positive costant, and 
let us take $\psi=\left[ T_h(u) 
\right]^{2k+1}$, where $T_h(s):=\min\left\{ 
s,h \right\}$, i.e. 
\begin{equation*}
\psi(x)=
\begin{cases}
u^{2k+1}(x) \ &\text{in } \left\{ x \in 
\Omega \ : \ u(x)\leq h \right\},\\
h^{2k+1} \ &\text{in }\left\{ x \in \Omega \ 
: \ u(x)> h \right\}.
\end{cases}
\end{equation*}
Thus, it results $\psi \in H^1(\Omega)$, and 
since $Tr(u)=0$ on $\Gamma_\Dcal$, 
$Tr(\psi)=Tr(u^{2k+1})=0$ on $\Gamma_\Dcal$, 
so that
we can take $\psi$ as test function in 
\eqref{problema mu}. Assumptions
\eqref{condizione di crescita H} and  
\eqref{(2.5) Montoro} yield
\begin{align*}
\begin{split}
\underline{\alpha}^2(2k+1)\int_{\left\{ u\leq 
h \right\}}|\nabla u|^2 u^{2k}&\leq(2k+1)
\int_{\left\{u\leq h\right\}}u^{2k}
\Hcal^2(\nabla u)\\
&=\int_{\left\{ u\leq h \right\}} 
\Hcal(\nabla u)\nabla_\xi \Hcal(\nabla u)
\cdot[(2k+1)u^{2k}\nabla u]\\
&=\int_{\Omega}  
\Hcal(\nabla u)\nabla_\xi \Hcal(\nabla u) 
\cdot \nabla \psi\\
&=\frac{1}{d} \left\{\int_\Omega mu\psi+
\mu(d,m)\int_\Omega u \psi\right\}\\
&\leq \frac{1}{d}\left[\| m \|_\infty+
\mu(d,m)\right] \int_\Omega u^{2(k+1)}.
\end{split}
\end{align*}
From now on, we can obtain an $L^\infty$ 
estimate for $u$ using the same argument in 
\cite[Theorem 3.1]{DG'}.
As a consequence, we can use the 
Harnack inequality proved in \cite[Theorem 
1.1]{Trudinger} to get that $u$ is positive 
in $\Omega$.
Finally, by applying Theorem \ref{regolarità}
we obtain $u \in C^{0,\gamma}(\overline{
\Omega})$, for some $\gamma \in (0,1/2)$.
\end{proof}
\begin{rmk}
Let us point out that the existence of $\mu $ has been obtained by minimization
on the cone of nonnegative functions due to  assumption \eqref{H pos omogenea}.
This restriction is not needed anymore if $ \Hcal(t\xi)=|t|\Hcal(\xi)$ for all
$\xi \in \mathbb{R}^N$ (see \cite{Pepisc} for more details with this respect).
\end{rmk} 

\begin{rmk}
Although the argument of \cite{DG'} can be
exploited to prove that every solution to the
nonlinear Problem \eqref{problema non lineare}
is bounded, we do not need it in this context,
as we obtain the unique solution 
constrained between two
bounded sub- and super-solutions as shown in the following result.
\end{rmk}
The sign of $\mu(d,m)$ acts as a threshold for the existence of a positive solution
to Problem \eqref{problema non lineare} as the next result shows.
\begin{theorem}\label{esistenza ellittico}
Let $\mu(d,m)$ be defined in \eqref{def mu}.
The following conclusions hold
\begin{itemize}
\item[(i)] if $\mu(d,m)<0$,
then there
exists a unique positive bounded weak 
solution 
to
Problem \eqref{problema non lineare}.
Moreover, $u \in C^{0,\gamma}(\overline{
\Omega})$, for some $\gamma \in (0,1/2)$;
\item[(ii)] if $\mu(d,m)\geq 0$, then the 
only
nonnegative weak solution to Problem
\eqref{problema non lineare} is the trivial 
one.
\end{itemize}
\end{theorem}
\begin{proof}
Let us first prove conclusion (i).
Let $\phi$ be the positive 
eigenfunction  associated with $\mu(d,m)$
normalized in $L^\infty$. Then, for all $\epsilon>0$ one 
has 
$v_\epsilon:=\epsilon\phi\equiv 0$ on $
\Gamma_\Dcal$, since $\phi \in 
H_\Dcal^1(\Omega)$.
Moreover, since $\mu(d,m)<0$,
by exploiting \eqref{limite f},
we can fix $\overline{\epsilon}>0$ such that
$$f(x,\epsilon\phi)\geq \epsilon\phi(m+
\mu(d,m)) \quad \text{for 
all }
0\leq \epsilon<\overline{\epsilon}.$$
Hence, for every $\epsilon\leq 
\overline{\epsilon}$, taking into account  \eqref{problema mu}
and  assumption \eqref{H pos omogenea}, we obtain
\begin{align*}
d\int_\Omega \Hcal(\nabla v_\epsilon)
\nabla_\xi \Hcal(\nabla v_\epsilon)\cdot 
\nabla w&=d\epsilon\int_\Omega 
\Hcal(\nabla \phi)\nabla_\xi \Hcal(\nabla 
\phi)\cdot \nabla w=\int_\Omega (m
+\mu(d,m))\epsilon\phi w \\
&\leq \int_\Omega 
f(x,v_\epsilon) w,
\end{align*}
for all $w \in 
C_\Dcal^\infty(\Omega)$, with $w\geq 0$. Namely,
 $v_\epsilon$ is a positive 
bounded weak sub-solution to Problem 
\eqref{problema non lineare}.
Let us consider the function $
\overline{u}:=M+1$, where $M$ is introduced in 
\eqref{f ii}. It is immediate to
observe that $\overline{u}$ 
is a positive bounded weak super-solution, 
and,
taking $\epsilon<\min\left\{ 
M+1, 
\overline{\epsilon} \right\}$, $\overline{u}
\geq 
\underline{u}$.
By applying 
Proposition \ref{propo sopra e sotto 
soluzioni} one obtains the existence of a 
weak solution 
$u \in H_\Dcal^1(\Omega)$ to Problem 
\eqref{problema non lineare}, 
with
$\underline{u}\leq u \leq \overline{u}$, so 
that
$u$ is positive and bounded.
Moreover, $u \in C^{0,
\gamma}(\overline{\Omega})$, for some $\gamma 
\in (0,1/2)$.
Let us now prove that $u$ is unique.
Let $u,v$ 
be positive, bounded, weak solutions to
Problem \eqref{problema non lineare}.
For every $\epsilon>0$, $\frac{u^2}{v+
\epsilon} \in H_\Dcal^1(\Omega)$,
and we can take it as test function in the 
equation 
solved by $v$. By applying \cite[Lemma 2.2]
{Jaros}, we obtain
\begin{align*}
0\geq d\int_\Omega \Hcal(\nabla v)
\nabla_\xi \Hcal(\nabla v)\cdot 
\nabla\left( \frac{u^2}{v+\epsilon}\right)-
d\int_\Omega \Hcal^2(\nabla u)=\int_\Omega 
f(x,v)\frac{u^2}{v+\epsilon}-\int_\Omega 
f(x,u)u.
\end{align*}
Then, taking $\frac{v^2}{u+\epsilon}$ as 
test funtion in the equation solved by $u$, 
and summing up, we get
\begin{align*}
0&\geq \int_\Omega f(x,v)\left[ 
\frac{u^2}{v+\epsilon}-v \right]
-\int_\Omega f(x,u)\left[ u-\frac{v^2}{u+
\epsilon} 
\right]\\
&=
\int_\Omega
\left\{
\frac{f(x,v)}{v+\epsilon}[
u^2-v(v+\epsilon)]-\frac{f(x,u)}{u+\epsilon}
[u(u+\epsilon)-v^2]\right\}.
\end{align*}
Taking into account \eqref{rmk coercività},
we can pass to the limit
as $\epsilon\rightarrow 0^+$ and
deduce 
\begin{align*}
0\geq \int_\Omega \left[ \frac{f(x,v)}{v}-
\frac{f(x,u)}{u} \right](u^2-
v^2)=&\int_{\Omega\cap \left\{ u\leq v 
\right\}} \left[ \frac{f(x,v)}{v}-
\frac{f(x,u)}{u} \right](u^2-v^2)\\
&+\int_{\Omega\cap \left\{ u > v 
\right\}} \left[ \frac{f(x,v)}{v}-
\frac{f(x,u)}{u} \right](u^2-v^2).
\end{align*}
Assumption \eqref{f iii} implies that both 
terms 
of the right-hand side are nonnegative,
yielding $u\equiv v$ a.e. in $\Omega$.\\
Let us now show conclusion (ii).
Let $u \in H_\Dcal^1(\Omega)$, be a nonnegative
weak solution to Problem \eqref{problema 
non lineare}, and suppose by contradiction that
$u\not\equiv 0$.  Let us fix $k> M$ and take as test function $v=T_{k}((u-M)^{+})$.
In view of \eqref{f ii} and \eqref{condizione di crescita H}, we obtain
\[
\underline{\alpha}\int_{\Omega}|\nabla T_{k}(u-M)^{+}|^{2}\leq \int_{\{x\in \Omega : u(x)>M\}}
f(x,u)T_{k}(u-M)\leq 0.
\]
Then, passing to the limit as $k\to+\infty$ we deduce that $u\leq M$ almost everywhere in $\Omega$.
Taking into account \eqref{rmk 
coercività} and taking $u$ as test function
, one has
$$0=d\int_\Omega \Hcal^2(\nabla u) -
\int_\Omega f(x,u)u > d\int_\Omega 
\Hcal^2(\nabla u)-\int_\Omega mu^2 \geq 
\mu(d,m)\int_\Omega u^2\geq 0,$$
yielding a contradiction.
\end{proof}
Theorem \ref{esistenza ellittico} gives us
a sufficient condition for existence and 
uniqueness of a positive stationary state
to Problem \eqref{problema parabolico}.
In the following we aim to prove
Theorem \ref{thm par}.
In order to do this,
let us introduce the functional setting. 
We consider the following
Hilbert triplet $$\left( 
H_\Dcal^1(\Omega) \ , \ L^2(\Omega) \ , \ 
(H_\Dcal^1(\Omega))^* \right),$$
where $(H_\Dcal^1(\Omega))^*$ is the dual 
space of $H_\Dcal^1(\Omega)$. 
We denote with $\langle\cdot , \cdot \rangle_*$
the dual product\\
$\langle \cdot , \cdot \rangle_{
(H_\Dcal^1(\Omega))^*,H_\Dcal^1(\Omega)}$,
and $\Omega_T=\Omega\times(0,T)$.
In order to prove the existence of a weak 
solution to
Problem \eqref{problema parabolico}
we will exploit sub- and super-solutions 
defined as
follows in the parabolic context (see e.g.
\cite{Di benedetto}). We will make use of the following problem
\begin{equation}\label{prob L}
\begin{cases}
\hskip68pt \mathcal{L}(v)=g(x,v)  \ &\text{in } 
\Omega_T,\\
\hskip63pt v(x,t)=0  &\text{on 
} \Gamma_\Dcal\times(0,T) ,\\
\Hcal(\nabla v)\nabla_\xi \Hcal(\nabla v)\cdot n=0
&\text{on } \Gamma_\Ncal\times(0,T),\\
\hskip61pt v(x,0)=v_0(x)
&\text{in }
\Omega,
\end{cases}
\end{equation}
where $\Lcal:L^2(0,T;H_\Dcal^1
(\Omega))\mapsto L^2(0,T;(H_\Dcal^1(\Omega))^*)
$ is defined as
\begin{equation}\label{L}
\Lcal(v):=\partial_t v- d \ \text{div}
(\Hcal(\nabla v)\nabla_\xi \Hcal(\nabla v))+\sigma v,
\end{equation}
$\sigma\geq 0$ will be properly chosen, and 
$g$ satisfies assumptions \eqref{f i} and 
\eqref{f iv}.
\begin{dfz}\label{def sotto sol 2}
A function $v \in L^2(0,T; H^1(\Omega))\cap 
C([0,T];L^2(\Omega))$ is a weak sub-solution 
to Problem \eqref{prob L}, if $
\partial_t v \in L^2(0,T;L^2(\Omega))$ and 
\begin{itemize}
\item[(i)]for every $\varphi \in 
H_\Dcal^1(\Omega)$, with $\varphi\geq 0$ 
, one has for a.e. $t \in (0,T)$
\begin{equation}\label{sotto sol para}
\begin{split}
\left( \partial_t v(t),
\varphi\right)_2 &+ d\left( 
\Hcal(\nabla v(t))\nabla_\xi \Hcal(\nabla 
v(t)),\nabla \varphi\right)_2 \\
&+\sigma \left( v(t),\varphi \right)_2
\leq \left( g(\cdot,v(t)),
\varphi\right)_2;
\end{split}
\end{equation}
\item[(ii)] $v(x,t)\leq 0$ on $
\Gamma_\Dcal\times 
(0,T)$;
\item[(iii)] $v(x,0)\leq v_0(x)$ a.e. in $
\Omega$.
\end{itemize}
The definition of super-solution is 
analogous, with the reverse inequalities.
\end{dfz}
A weak solution to 
Problem \eqref{problema parabolico} is a function that is both a weak sub-
and super-solution. More precisely, 
taking $\varphi \in H_\Dcal^1(\Omega)$, 
testing \eqref{sotto sol para} with $\varphi^+,
\varphi^- \in 
H_\Dcal^1(\Omega)$, and then subtracting, we 
obtain the following definition.
\begin{dfz}\label{def 2}
A function $v \in L^2(0,T; H_\Dcal^1(\Omega))
\cap 
C([0,T];L^2(\Omega))$ is a weak solution 
to Problem \eqref{prob L}, if $
\partial_t
v \in L^2(0,T;L^2(\Omega))$, $v(x,0)=v_0(x)$
a.e. in $\Omega$,
and
\begin{equation}\label{sol para}
\begin{split}
\left( \partial_t v(t),
\varphi\right)_2 &+ d\left( 
\Hcal(\nabla v(t))\nabla_\xi \Hcal(\nabla 
v(t)),\nabla \varphi\right)_2\\
&+\sigma\left( v(t),\varphi 
\right)_2
= \left( g(\cdot,v(t)),
\varphi\right)_2,
\end{split}
\end{equation}
for every $\varphi \in 
H_\Dcal^1(\Omega)$,
and for a.e. $t \in (0,T)$.
\end{dfz}
\begin{rmk}\label{def eq}
Thanks to \eqref{condizione di 
crescita H} and \eqref{(2.6) Montoro},
one can use classical arguments  to deduce that
Definition \ref{def 2} is equivalent to
require that $v$ satisfies 
\begin{equation}
\begin{split}
\int_0^T\!\! \left( \partial_t v(t),
\varphi(t) \right)_2dt &+ 
d\int_0^T\!\! \left( \Hcal(\nabla v(t))
\nabla_\xi \Hcal(\nabla v(t)),\nabla 
\varphi(t)\right)_2dt\\
&+\sigma \int_0^T \left( v(t),\varphi(t) 
\right)_2 = 
\int_0^T\!\! \left( g(\cdot,v(t)),\varphi(t)
\right)_2dt,
\end{split}
\end{equation}
for every $\varphi \in L^2(0,T; 
H_\Dcal^1(\Omega))$.
\end{rmk}
To prove existence and uniqueness of a weak solution 
to
Problem \eqref{problema parabolico}, we will 
exploit
comparison principles.
Let us first prove the following result.
\begin{lemma}\label{lemma unicità} 
Let $\underline{v}, \overline{v}
\in C([0,T]; L^2(\Omega))\cap 
L^2(0,T;H^1(\Omega))$ be respectively 
bounded 
weak sub- and super-solutions to 
Problem 
\eqref{prob L}, 
then for all $t \in [0,T]$ it results 
$$
\underline{v}(t) 
\leq \overline{v}(t) \ 
\ \text{a.e.  in} \ \Omega.$$
\end{lemma}
\begin{proof}
First of all, observe that $a\leq 
\underline{v}$, $\overline{v}\leq 
b$ a.e. in $\Omega_T$, for some $a,b 
\in 
\mathbb{R}$, since $\underline{v},
\overline{v}$ are bounded.
Since $(\underline{v}-
\overline{v})^+ \in C([0,T]; 
L^2(\Omega))\cap 
L^2(0,T;H_\Dcal^1(\Omega))$, and $(
\underline{v}-\overline{v})^+(t)
\geq 0$, we can take it as test function in the equations satisfied by $\underline{v}$, $\overline{v}$. We obtain
\begin{align}\label{per avere ordinamento}
\begin{split}
 & \left( \partial_t (\underline{v}-
 \overline{v})(t),(\underline{v}-
 \overline{v})^+(t)
 \right)_2 
+\sigma \left( (\underline{v}-\overline{v})(t),
(\underline{v}-\overline{v})^+(t) 
\right)_2\\
& \;+ d \left( \Hcal(\nabla \underline{v}(t))
\nabla_\xi 
\Hcal(\nabla \underline{v}(t))-\Hcal(\nabla 
\overline{v}(t))
\nabla_\xi \Hcal(\nabla \overline{v}(t)) ,
\nabla(\underline{v}-\overline{v})^+
(t)\right)_2\\
 &\leq \left( g(\cdot,
\underline{v}(t))-g(\cdot,\overline{v}(t)),(
\underline{v}-\overline{v})^+(t) 
\right)_2, \quad \text{for a.e. }t 
\in (0,T).
\end{split}
\end{align}
Taking into account that $\sigma\geq 0$, that
\begin{align*}
\left( \partial_t (
\underline{v}-\overline{v})(t),(
\underline{v}-\overline{v})^+(t)
\right)_2=\frac{1}{2}\frac{d}{dt}
\| (\underline{v}-\overline{v})^+(t)\|_2^2,
\end{align*}
and that, thanks to \eqref{hessiano H}, one has
$$\left( \Hcal(\nabla 
\underline{v}(t))\nabla_\xi 
\Hcal(\nabla \underline{v}(t))-\Hcal(\nabla 
\overline{v}(t))
\nabla_\xi \Hcal(\nabla 
\overline{v}(t)) ,\nabla(\underline{v}-
\overline{v})^+
(t)\right)_2\geq 0,$$
inequality
\eqref{per avere ordinamento} yields
$$\frac{1}{2}\frac{d}{dt}\| (
\underline{v}-\overline{v})^+(t)\|_2^2 
\leq \left( g(\cdot,
\underline{v}(t))-g(\cdot,\overline{v}(t)),(
\underline{v}-\overline{v})^+(t) 
\right)_2.$$
On the other hand, in view of 
\eqref{f i} we can write
$$g(\cdot,\underline{v}(t))-g(\cdot,\overline{v}
(t))=\int_0^1 \frac{d}
{dq} g(\cdot,q(
\underline{v}-\overline{v})+\overline{v})dq=(
\underline{v}-\overline{v})\int_0^1 \partial_s 
g(\cdot, q(\underline{v}-\overline{v})+
\overline{v})dq.$$
Hence, setting 
$$r(\cdot,t):=\int_0^1 \partial_s g(\cdot, q(
\underline{v}-\overline{v})
+\overline{v})dq,$$
from assumption \eqref{f iv} it follows
\begin{equation}\label{per rmk}
\frac{1}{2}\frac{d}{dt}\| (
\underline{v}-\overline{v})^+(t) \|_2^2\leq 
\int_\Omega r(\cdot,t)[(
\underline{v}-\overline{v})^+(t)]^2 
dx \leq L \| (\underline{v}-
\overline{v})^+(t) \|_2^2.
\end{equation}
Recalling that $\underline{v}(x,0)-\overline{v}
(x,0)\leq v_0(x)-
v_0(x)=0$, and by applying Gronwall Lemma
we obtain the conclusion.
\end{proof}
We are now in a position to prove the following.
\begin{theorem}\label{teorema esistenza par }
Let $\underline{v}$ and $\overline{v}$ be 
respectively bounded weak sub- and super-
solutions to Problem \eqref{problema 
parabolico}.
Then, there exists $v \in 
L^2(0,T;H_\Dcal^1(\Omega))\cap
C([0,T];L^2(\Omega))$ unique weak 
solution to Problem \eqref{problema 
parabolico} such that $$\underline{v}\leq v 
\leq \overline{v} \ \ \text{a.e. in } \ 
\Omega_T.$$
\end{theorem}
\begin{proof}
We adapt to our case
a classical iteration scheme by
\cite{Sattinger}. 
Since $\underline{v}$ and $\overline{v}$ are 
bounded, by applying Lemma \ref{lemma unicità} we have
$a\leq \underline{v}\leq 
\overline{v}\leq b$ a.e. in $\Omega_T$, for some
$a,b \in \mathbb{R}$. We 
set $F(x,s):=f(x,s)+Ls$, where 
$L$ is given in \eqref{f iv}, so that 
$$\partial_s F(x,s)=\partial_s f(x,s)+L\geq 
-L+L=0,$$
and we denote by $v_1:=\underline{v}$. 
From assumption \eqref{f iv} it follows that
$F(x,v_1)=
F(x,\underline{v}) \in L^2(0,T;L^2(\Omega))$, 
and 
knowing that
$L^2(\Omega)$ is continuously embedded in $
(H_\Dcal^1(\Omega))^*$, one can apply 
\cite[Theorem 1.2, Chap. 2]{Lions}
to 
find a 
unique weak solution $v_2 \in 
L^2(0,T; H_\Dcal^1(\Omega))$ to problem
\begin{equation}\label{prob L 1}
\begin{cases}
\hskip73pt\Lcal(v)=F(x,v_1) \ 
&\text{in } \Omega_T,\\
\hskip89pt v=0  &\text{on} \ 
\Gamma_\Dcal\times (0,T),\\ 
\hskip4pt \Hcal(\nabla v)\nabla_\xi 
\Hcal(\nabla v)\cdot n=0 &\text{on } 
\Gamma_\Ncal \times (0,T),\\
\hskip66pt v(x,0)=v_0(x) &\text{in} \ 
\Omega,
\end{cases}
\end{equation}
where $\Lcal$ is defined in \eqref{L}, with
$\sigma=L$,
and 
$v_2 \in C([0,T]; L^2(\Omega))$.
Moreover, conditions \eqref{condizione di 
crescita H}, \eqref{hessiano H} and
\eqref{(2.5) Montoro} allow us to apply
\cite[Proposition 4.2, Chap. III]{Showalter} and obtain that
$\partial_t 
v_2 \in L^2(0,T; L^2(\Omega))$.
In particular, this implies that $v_1$ and $v_2$
are respectively weak sub- and super-solutions to
problem \eqref{prob L 1},
with
$g(x,v)=g(x)=F(x,v_1(x))$, and Lemma
\ref{lemma unicità} yields
$v_1\leq v_2$ a.e.
in $\Omega_T$.
In addition, the monotonicity 
property of $F$ implies that $\Lcal(\overline{v})
\geq 
F(x,\overline{v})\geq \Lcal(v_2)$,
so that
$v_2$ and $\overline{v}$ are respectively weak 
sub- and super-solutions to Problem \eqref{prob L 1},
thus,
by applying again Lemma \ref{lemma unicità} we 
have
$$a\leq v_1\leq v_2\leq \overline{v}\leq b 
\ \ \text{ a.e. in}\ \ \Omega_T.$$
Now, we proceed by induction. Assume $k>1$ 
and  $v_{k-1} \in C([0,T]; 
L^2(\Omega))\cap L^2(0,T;H_\Dcal^1(\Omega))$, 
with $\partial_t v_{k-1}\in L^2(0,T; 
L^2(\Omega))$, such that 
\begin{equation}\label{eq:vklim}
a\leq v_1\leq v_2\leq \cdots \leq v_{k-2}
\leq v_{k-1}\leq \overline{v}\leq b \ \ 
\text{a.e. in}\ \ \Omega_T.
\end{equation}
By applying 
\cite[Theorem 1.2 Chap. 2]{Lions} and \cite[Proposition 4.2, Chap.
III]{Showalter}
again, we obtain the existence of 
$v_k \in C([0,T]; L^2(\Omega))\cap 
L^2(0,T;H_\Dcal^1(\Omega))$, with $\partial_t v_k\in 
L^2(0,T;L^2(\Omega))$,
 unique weak solution to problem
\begin{equation}\label{prob k}
\begin{cases}
\hskip73pt \Lcal(v)=F(x,v_{k-1}) \ 
&\text{in } 
\Omega_T,\\
\hskip89pt v=0 &\text{on } 
\Gamma_\Dcal\times 
(0,T),\\ 
\hskip4pt \Hcal(\nabla v)\nabla_\xi \Hcal(\nabla v)
\cdot n=0 &\text{on } \Gamma_\Ncal \times 
(0,T),\\
\hskip66pt v(x,0)=v_0(x) &\text{in } 
\Omega.
\end{cases}
\end{equation}
Moreover, $\Lcal(v_{k-1})=F(x,v_{k-2})\leq F(x,v_{
k-1})=\Lcal(v_k)$,
then, arguing as above, we have constructed
a sequence
$\left\{ v_k \right\}\in 
L^2(0,T;H_\Dcal^1(\Omega))$, with 
$\partial_t v_k\in L^2(\Omega_T)$, satisfying
\eqref{eq:vklim} and \eqref{prob k}.
Let us prove now that $\left\{ v_k \right\}$ 
is uniformly bounded in 
$L^2(0,T;H_\Dcal^1(\Omega))$. It results
\begin{align*}
\left(\partial_t v_k(t),v_k(t) 
\right)_2+ & d \left( 
\Hcal(\nabla v_k(t))\nabla_\xi \Hcal(\nabla 
v_k(t))), \nabla v_k(t) \right)_2
\\
&+L\| v_k(t)\|_2^2=\left( 
F(\cdot,v_{k-1}),v_k(t) \right)_2
\ \ \text{for a.e. } \ t \in (0,T).
\end{align*}
Integrating in $(0,T)$, and recalling 
assumptions \eqref{condizione di crescita H}, \eqref{(2.5) Montoro}, and Proposition
\ref{poincare}, we obtain
\begin{small}
\begin{equation*}\label{per lim di u_k}
\int_0^T\!\!\left(\partial_t v_k(t),v_k(t) 
\right)_2dt + d 
\underline{\alpha}^2  \int_0^T\!\! \| 
v_k(t) 
\|
^2dt +L\int_0^T\!\! \| v_k(t) \|_2^2 dt \leq 
\int_0^T\!\! \left( F(\cdot,v_{k-1}),v_k(t) 
\right)_2 dt.
\end{equation*}
\end{small}
Since
\begin{small}
$$\int_0^T\!\!\left(\partial_t v_k(t),v_k(t) 
\right)_2dt=\frac{1}{2}\| v_k(T) 
\|_2^2-\frac{1}{2}\| v_0 \|_2^2,
\quad
\int_0^T\!\! \left( F(\cdot,v_{k-1}),v_k(t) 
\right)_2 dt \leq b \int_0^T\!\! 
\int_\Omega F(x,b)dx dt,
$$
\end{small}
one deduces
$$\frac{1}{2}\| v_k(T) \|_2^2+d 
\underline{\alpha}^2 \int_0^T\!\! \| v_k(t)\|
^2 dt+L\int_0^T\!\! \| v_k(t)\|_2^2 dt\leq b 
\int_{\Omega_T}F(x,b)dxdt+\frac{1}{2}\|  
v_0\|_2^2,$$
so that, in particular
$$\int_0^T\!\! \| v_k(t)\|^2 dt\leq \frac{1}
{d \underline{\alpha}^2} 
\left\{b\int_{\Omega_T}F(x,b)dx dt+\frac{1}
{2}\| v_0 \|_2^2\right\}.$$
implying that
$\left\{ v_k \right\}_{k\geq 1}$ is uniformly 
bounded in $L^2(0,T;H_\Dcal^1(\Omega))$.
Then, there exists 
$v \in L^2(0,T;H_\Dcal^1(\Omega))$ such that, up to a subsequence, 
$v_k\rightharpoonup v$ weakly in $L^2(0,T;H_\Dcal^1(\Omega))$, 
$v_k \rightarrow v$ strongly in $L^2(\Omega_T)$ and 
a.e. in $\Omega_T$.
Let us  prove that $\nabla v_k 
\rightarrow \nabla v$ in 
$L^2(\Omega_T)$. 
Testing \eqref{prob k} with $v_k-v_l \in 
L^2(0,T;H_\Dcal^1(\Omega))$ and denoting 
$$\Acal_{k,l}(t):=\Hcal(\nabla v_k(t))\nabla_\xi 
\Hcal(\nabla v_k(t))-\Hcal(\nabla v_l(t))
\nabla_\xi \Hcal(\nabla v_l(t)),$$ 
we have
\begin{align}\label{per conv gradienti}
\begin{split}
\frac{1}{2}\| v_k(T)-v_l(T)\|_2^2 +&L\int_0^T  \|v_k(t)-
v_l(t)\|_2^2  dt\\
+& d \int_0^T\!\! \left(\Acal_{k,l}(t),
 \nabla (v_k(t)-v_l(t))\right)_{L^2}  dt\\
&= \int_0^T\!\! \left(F(x,v_{k-1}(t))-
F(x,v_{l-1}(t)),v_{k-1}(t)-v_{l-1}(t)\right)_{L^2} dt.
\end{split}
\end{align}
As $v_k$ converges in $L^2(\Omega_T)$,
assumption \eqref{f iv} allows us
to
apply Lebesgue Theorem on the right-hand side
and obtain
\begin{equation*}
\lim_{k,l \rightarrow +\infty}\int_0^T\!\! 
\int_\Omega [F(x,v_{k-1}(t))-F(x,v_{l-1}(t))]
(v_k(t)-v_l(t))dx dt=0.
\end{equation*}
Then, assumption \eqref{hessiano H} implies
$$0\leq \lim_{k,l \rightarrow +\infty} 
\Xi \int_0^T\!\! \int_\Omega |\nabla v_k(t)-
\nabla v_l(t)|^2 dx dt \leq \lim_{k,l 
\rightarrow +\infty} \int_0^T \left(
\Acal_{k,l}(t), \nabla (v_k(t)-v_l(t))\right) 
dt=0,$$
so that
$\nabla v_k\rightarrow \nabla v
\ \text{in} \ L^2(\Omega_T)$, and
$\nabla v_k(t)\rightarrow \nabla v(t)$
a.e. in $\Omega$ (up to a sub-sequence).
Recalling that $\Hcal \in C^2(\mathbb{R}^N
\setminus\left\{ 0 \right\})$, we have
$$\Hcal(\nabla v_k(t))\nabla_\xi
\Hcal(\nabla v_k(t))\rightarrow
\Hcal(\nabla v(t))\nabla_\xi
\Hcal(\nabla v(t)) \quad \text{a.e. in } 
\Omega.$$
Moreover, in view of  assumptions \eqref{condizione di 
crescita H}, \eqref{(2.5) Montoro},  we can apply  Lebesgue Theorem to have,
\begin{equation}\label{conv parte H}
\left( 
\Hcal(\nabla v_k(t))\nabla_\xi
\Hcal(\nabla v_k(t)),\nabla
\varphi \right)_2
\rightarrow
\left( 
\Hcal(\nabla v(t))\nabla_\xi
\Hcal(\nabla v(t)),\nabla
\varphi \right)_2,
\end{equation}
for all $\varphi \in H_\Dcal^1(\Omega)
$ and 
for a.e. $t \in (0,T)$.
Let us prove that $\left\{ 
\partial_t v_k \right\}_{k \geq 1}$ is 
uniformly bounded in $L^2(0,T;
(H_\Dcal^1(\Omega))^*)$.\\
For all $w \in 
H_\Dcal^1(\Omega)$ and a.e. $t \in (0,T)$, one 
has
$$(\partial_t v_k(t),w)_2=-d 
\left( \Hcal (\nabla v_k(t))\nabla_\xi 
\Hcal(\nabla v_k(t)), \nabla w 
\right)_2+\left( 
F(\cdot,v_{k-1}),w \right)_2.$$
Recalling that $\left\{ \|\nabla v_k(t) \|_2 
\right\}$ is bounded by a constant,
we obtain, for 
all $w \in H_\Dcal^1(\Omega)$ with $\| w \|=1$,
\begin{align*}
\left| (\partial v_k(t),w)_2 
\right| &\leq d \overline{\alpha}Q 
\int_\Omega |\nabla v_k(t)||\nabla w| dx + \| 
F(\cdot,v_{k-1})\|_2 \| w \|_2 \\
& \leq \frac{d \overline{\alpha}Q}{2}\| 
\nabla v_k(t) \|_2 + \|
F(\cdot,b) \|_2 
\leq \overline{C}.
\end{align*}
As a consequence, $\left\{\partial_t v_k 
\right\}$ is 
uniformly bounded in $L^2(0,T;
(H_\Dcal^1(\Omega))^*)$. Hence,
there exists $z$ in $L^2(0,T;
(H_\Dcal^1(\Omega))^*)$ such that,
up to a sub-sequence, one has $\partial_t v_k 
\rightharpoonup z$. It results $z=\partial_t 
v$. Indeed,
for every
$\varphi \in C_0^\infty(0,T)$ and for every
$w \in H_\Dcal^1(\Omega)$ we have
\[
\begin{split}
\int_0^T\!\! \langle z,w 
\rangle_* 
\varphi(t)dt &
=\lim_{k\to+\infty}\int_0^T\!\! \langle 
\partial_t v_k(t),w 
\rangle_* \varphi(t)dt=-\lim_{k\to+\infty}
\int_0^T\!\! 
(v_k(t),w)_2\partial_t 
\varphi(t)
\\&
=-\int_0^T\!\! 
(v(t),w)_2\partial_t 
\varphi(t)
=\int_0^T\!\! 
\langle \partial_t v(t),w 
\rangle_* \varphi(t),
\end{split}
\]
yielding that
$z=\partial_t v$.
Then, passing to the limit as $k 
\rightarrow +\infty$ in 
$$\left( \partial_t v_k, w 
\right)_2+d 
\left( \Hcal (\nabla v_k(t))\nabla_\xi 
\Hcal (\nabla v_k(t)), \nabla w 
\right)_2+L\left( v_k(t),w 
\right)_2=\left( 
F(\cdot,v_{k-1}),w \right)_2$$
we deduce
that $v$ satisfies \eqref{sol para}.\\
Since $v \in L^2(0,T;H_\Dcal^1
(\Omega))$, and $\partial_t v \in L^2(0,T;
(H_\Dcal^1
(\Omega)^*)$,  $v\in C([0,T]; L^2(\Omega))$; 
finally 
applying again \cite[Proposition 4.2, Chap.
III]{Showalter} we obtain $\partial_t v \in 
L^2(\Omega_T)$.
Uniqueness follows from Lemma \ref{lemma 
unicità}. Indeed, since $v_1,v_2$ are both
weak sub- and super-solutions
to Problem \eqref{problema parabolico}, by 
applying 
Lemma \ref{lemma 
unicità}, with $\sigma=0$ and $g(x,v)=f(x,v)$, 
we 
obtain that,
for every $t \in [0,T]$ it results
$v_1(t)\leq v_2(t)$ and $v_2(t)\leq v_1(t)$ a.e. 
in 
$\Omega$, so that
$v_1(t)\equiv v_2(t)$ a.e. in $
\Omega$.
\end{proof}
In view of Theorem 
\ref{teorema esistenza par }, in order to show 
the existence of a solution to Problem
\eqref{problema parabolico}, we just need
to find a pair of ordered bounded sub- and
super-solutions, as we will see in the proof
of the following result.
\begin{theorem}\label{es par mu}
There exists $v(x,t)\in L^2(0,T;
H_\Dcal^1(\Omega))\cap C([0,T];L^2(\Omega))$
unique nonnegative weak solution to Problem
\eqref{problema parabolico}. Moreover, the
following statements hold.
\begin{itemize}
\item[$(i)$]
If $\mu(d,m)<0$, then $v$ is positive, and
\begin{equation}\label{conver par}
\lim_{t \rightarrow +\infty}
v(x,t)=u(x),
\end{equation}
where $u$ is the positive solution found in
Theorem \ref{esistenza ellittico}.
\item[$(ii)$]
If $\mu(d,m)\geq 0$, then
\begin{equation}\label{conv par 0}
\lim_{t \rightarrow + \infty }v(x,t)=0.
\end{equation}
\end{itemize} 
\end{theorem}
\begin{proof}
Let us first suppose that $\mu(d,m)<0$.
Let $\phi$ be the positive eigenfunction 
associated with $\mu(d,m)$ defined in
\eqref{def mu}, normalized in $L^\infty$. We take 
$\epsilon\leq 
\epsilon_1:=\min_{\Omega}v_0$,
then, setting $v_\epsilon(x,t)=v_\epsilon(x):=
\epsilon \phi(x)$, one has 
$v_\epsilon(x,0)\leq v_0(x)$. Moreover,
$v_\epsilon\equiv 0$ on $
\Gamma_\Dcal$, since $\phi \in 
H_\Dcal^1(\Omega)$. 
In addition, since $\mu(d,m)<0$,
thanks to \eqref{limite f},
there exists $\epsilon_2>0$ such that 
$$f(x,u)\geq (m+\mu(d,m))u \ \ \text{for 
all} \ 0\leq u < \epsilon_2.$$
Hence, taking $\epsilon\leq \epsilon_2$ we 
have
\begin{align*}
d\int_\Omega \Hcal(\nabla v_\epsilon)
\nabla_\xi \Hcal(\nabla v_\epsilon)\cdot 
\nabla \varphi=d\epsilon\int_\Omega 
\Hcal(\nabla \phi)\nabla_\xi \Hcal(\nabla 
\phi)\cdot \nabla \varphi &=\int_\Omega (m
+\mu(d,m))\epsilon\phi\varphi \\
&\leq \int_\Omega 
f(x,v_\epsilon)\varphi,
\end{align*}
for all $\varphi \in 
H_\Dcal^1(\Omega)$.
Then,
for every
$\epsilon<\min\left\{ \epsilon_1,
\epsilon_2
\right\}$, $\underline{v}:=v_\epsilon$ is a 
time
independent weak sub-solution to Problem 
\eqref{problema parabolico},
while $\overline{v}=\max\left\{ 
M+1,\|v_0 \|_\infty \right\}$ is a 
time independent weak super-solution. Thus, by
applying Theorem \ref{teorema esistenza par }
we have proved that $v$ exists and it is 
positive. 
If $\mu(d,m)\geq 0$,
Theorem
\ref{esistenza ellittico} yields that
the only nonnegative stationary state to
Problem \eqref{problema parabolico} is 0, 
thus, taking $\underline{v}\equiv 0$ and again
$\overline{v}=\max\left\{ 
M+1,\|v_0 \|_\infty \right\}$ in Theorem 
\ref{teorema esistenza par }, we get the 
existence of $v$ nonnegative.
Let us prove now
\eqref{conver par} and \eqref{conv par 0}.
Let $\underline{v}$ be the time-independent 
weak sub-solution to Problem
\eqref{problema parabolico} given by
$0$ if $\mu(d,m)\geq 0$ and by $\epsilon\phi$
if $\mu(d,m)<0$.
Let $v_1,v_2$ be the positive solutions to
Problem \eqref{problema parabolico}
with initial datum $\underline{v}$
and $\overline{v}$ respectively.
Since $\underline{v}\leq v_0\leq 
\overline{v}$, 
exploiting Lemma \ref{lemma unicità} and
Theorem \ref{teorema esistenza par }
we get that
the solution $v$
to Problem \eqref{problema parabolico}
satisfies $\underline{v}\leq v_1\leq v \leq v_2
\leq \overline{v}$ a.e. in $\Omega_T$. 
Moreover, $v_1$ is increasing with respect to
$t$ and $v_2$ is decreasing with respect to 
$t$, let us show it for $v_1$. We fix $\delta>0$, 
we take
$h(x)=v_1(x,\delta)$ as initial datum, and we
denote with $v_h$ the corresponding solution to
Problem \eqref{problema parabolico}. Then,
$v_1 \leq v_h$ and $v_h(x,t)=v_1(
x,t+\delta)$ by uniqueness, hence
$$v_1(x,t)\leq v_1(x,t+\delta) 
\ \text{a.e. in } \Omega\times [0,+\infty).$$
The monotonicity of $v_1$ and $v_2$
yields the
existence of $w_1,w_2 \in H_\Dcal^1
(\Omega)$ such that 
$$v_1(x,t) \rightarrow w_1(x) \ , \ v_2(x,t)
\rightarrow w_2(x) \quad \text{as } t 
\rightarrow
+\infty, \ \text{for a.e. } x \in \Omega.$$
Let us consider the sequence $\left\{ 
v_k\right\}$ such that $v_k(x,t)=v_1(x,t+k)$,
with $0<t<1$,
then $v_k \in L^2(0,1; H_\Dcal^1(\Omega))
\cap H^1(0,1;L^2(\Omega))$
and it satisfies
\begin{equation}\label{(A)}
\begin{cases}
\partial_t v_k - d \text{div}
(\Hcal(\nabla v_k)\nabla_\xi
\Hcal(\nabla v_k))=f(x,v_k) \ & \text{in }
\Omega_1,\\
\hskip135pt v_k=0 & \text{on } 
\Gamma_\Dcal\times (0,1),\\
\hskip46pt
\Hcal(\nabla v_k)\nabla_\xi \Hcal(\nabla v_k)
\cdot n=0 & \text{on } \Gamma_\Ncal\times
(0,1),\\
\hskip111pt v_k(x,0)=v_1(x,k).
\end{cases}
\end{equation}
Then
$$\underline{v}\leq v_{k-1}\leq v_k
\leq v_{k+1}\leq \dots\leq \overline{v} \ 
\text{a.e. 
in } \Omega_T,$$
and $v_k \rightarrow w_1$ in
$L^2(\Omega_T)$. 
Arguing as in the proof of Theorem 
\ref{teorema esistenza par } we deduce 
$\partial_t v_k
\rightharpoonup \partial_t w_1\equiv 0$
in $L^2(\Omega_T)$. Thus, we can pass to the 
limit as $k \rightarrow 0$ in the weak
formulation of \eqref{(A)} to obtain that
$w_1$ is a weak solution to Problem
\eqref{problema non lineare}, and by
its uniqueness, $w_1\equiv u$.
Analogously one can prove that
$w_2\equiv u$, yielding both
\eqref{conver par} and \eqref{conv par 0}.
\end{proof}
\section{Proof of the Main Results}\label{lambda}
This section is devoted to the proof of Theorems
\ref{esistenza sotto d star}, 
\ref{thm par}, \ref{minimo minimo} and
\ref{D intervallo}.
\\
Let us first prove that $\lambda(m)$ defined 
in \eqref{def lambda} is a positive 
minimum. To this aim,  we will suppose that $m\in L^{\infty}(\Omega)$ and 
$m^{+}\not \equiv0$; indeed, while the existence of $\mu(d,m)$ holds for every $m\in L^{\infty}(\Omega)$, the existence of $\lambda(m)$ and its positiveness require that $m^{+}\not \equiv0$.
\begin{proposition}\label{lambda minimo}
Let $m \in L^{\infty}(\Omega)$, with $m^{+}\not \equiv0$.
Then,  $\lambda(m)$, defined in  \eqref{def lambda}, is
attained by a unique positive
eigenfunction $\varphi \in H_\Dcal^1(\Omega)
$, up to a multiplication by a constant. 
Moreover, $\varphi \in
C^{0,\gamma}(\overline{\Omega})$, for some $
\gamma \in (0,1/2)$, and $\lambda(m)$ is the 
unique positive eigenvalue with positive 
eigenfunction.
\end{proposition}
\begin{proof}
First, we observe that, thanks to assumption 
\eqref{H pos omogenea}, the quotient in \eqref{def lambda}
 is invariant up to a  positive normalization of u, thus we can find $
\lambda(m)$ by solving the following 
constraint minimization problem
$$\lambda(m)=\inf\left\{ \int_\Omega 
\Hcal^2(\nabla u) \ : \ u \in 
H_\Dcal^1(\Omega), \ u \geq 0 , \ \int_\Omega 
mu^2=1 \right\}.$$
The set we are minimizing on is not empty, 
indeed, having observed that $\chi_{
\left\{m>0\right\}} \in L^2(\Omega)$, 
there 
exists a sequence $\left\{ \psi_n \right\} 
\subset C_c^\infty(\Omega)$ such that $
\psi_n \rightarrow \chi_{\left\{m>0\right\}}$ 
in $L^2(\Omega)$, hence 
\begin{align*}
\int_\Omega m\psi_n^2 &=\left\{ 
\int_\Omega m\chi_{\left\{m>0\right\}}^2 
+\int_\Omega m\left( \psi_n^2-
\chi_{\left\{m>0\right\}}^2 \right) 
\right\}
\geq \int_{\left\{ m>0 \right\}}m-\| m 
\|_\infty\int_\Omega \left| \psi_n^2 -
\chi_{\left\{m>0\right\}}^2 \right|>0,
\end{align*}
for $n$ sufficiently large. This implies that
$\lambda(m)<\infty$.\\
We set $\mathcal{I}=\left\{ u \in 
H_\Dcal^1(\Omega)\ : \ u \geq 0 , \ 
\int_\Omega 
mu^2=1 \right\}$.
Let $\left\{ \varphi_n \right\}
\subset 
\Ical$ be a minimizing sequence, then, 
thanks to \eqref{condizione di crescita H},
we get that
$\left\{ \varphi_n \right\}$
is bounded in $H_\Dcal^1(\Omega)$, so that there 
exists $\varphi \in H_\Dcal^1(\Omega)$ such 
that $\varphi_n 
\rightharpoonup \varphi$ in 
$H_\Dcal^1(\Omega)
$. 
By the compact embedding it follows that, up to 
a 
sub-sequence, $\varphi_n \rightarrow 
\varphi$ a.e. and in $L^2(\Omega)$, yielding
$$1=\lim_{n\to +\infty} \int_\Omega m
\varphi_n^2=\int_\Omega m \varphi^2.$$
Then $\varphi \in \Ical$ and $$
\int_\Omega \Hcal^2(\nabla \varphi)\geq 
\lambda(m).$$ 
On the other hand, by 
the weak lower semicontinuity of
the $L^2$-norm of 
$\Hcal(\nabla \cdot)$
it follows that $$\int_\Omega 
\Hcal^2(\nabla 
\varphi)\leq \liminf_{n\rightarrow +\infty}  
\int_\Omega \Hcal^2(\nabla \varphi_n)= 
\lambda(m),$$
hence $$\int_\Omega \Hcal^2(\nabla 
\varphi)=\lambda(m).$$
In particular, it is easy to check that
$\varphi$ is a weak solution 
to Problem \eqref{problema anisotropo}.
Now, arguing as in \cite{DG'} we 
obtain that $\varphi \in L^\infty(\Omega)$, 
thus we can use the Harnack inequality 
proved in 
\cite[Theorem 1.1]{Trudinger} to obtain 
that $\varphi$ is positive in $\Omega$, and 
by applying Theorem \ref{regolarità} one 
has 
$\varphi \in C^{0,\gamma}(\overline{\Omega})
$, 
for some $\gamma\in (0,1/2)$.
To prove that the positive eigenfunction is 
unique and that $\lambda(m)$ is the only 
eigenvalue with positive eigenfunction, one 
can argue as in \cite{Pepisc}. 
\end{proof}
Proposition \ref{lambda minimo} ensures
that Problem \eqref{problema anisotropo} 
actually admits a unique positive principal 
eigenvalue, thus from now on $\lambda(m)$ 
will denote the positive principal 
eigenvalue associated with Problem 
\eqref{problema anisotropo}.\\
We are now in a position to conclude the proofs 
of 
Theorems \ref{esistenza sotto d star} and
\ref{thm par}.
\begin{proof}[Proof of Theorems
\ref{esistenza sotto d star} and \ref{thm par}]
Let us fix 
\[d^*=\frac{1}{\lambda(m)}=\dfrac{\int_{\Omega} m\varphi^{2}}{\int_{\Omega}\Hcal^{2}(\nabla \varphi)},
\]
where $\varphi\in H_\Dcal^1(\Omega)$ be the
positive, bounded eigenfunction associated
with $\lambda(m)$, found thanks to Proposition
\ref{lambda minimo}.
We will prove that 
$\mu(d,m)<0$ if and only if 
$0<d<d^*$.
Since $m$ is fixed in this argument, we omit the dependence of $m$ in $\mu$ and we denote $\mu(d,m)=
\mu(d)$.
If $d<d^*$, then
$$d\int_\Omega \Hcal^2(\nabla \varphi)
-\int_\Omega m\varphi^2<0,$$
so that
$$\mu(d)\leq \frac{d\int_\Omega \Hcal^2(\nabla 
\varphi)
-\int_\Omega m\varphi^2}{\int_\Omega \varphi^2}
<0.$$
On the other hand, if $\mu(d)<0$, taking $\phi$ 
as the 
positive, bounded 
eigenfunction associated with $\mu(d)$,
found thanks to Proposition 
\ref{l'autovalore è minimo}, it results
$$d\int_\Omega \Hcal^2(\nabla \phi)-
\int_\Omega m\phi^2<0,$$
so that $\int_\Omega m \phi^2>0$ and 
$$d<\frac{1}{\frac{\int_\Omega \Hcal^2(\nabla 
\phi)}{\int_\Omega m\phi^2}}\leq \frac{1}{\lambda
(m)}=d^*.$$
Then,  
the proof of Theorem \ref{esistenza sotto d 
star} can be concluded by applying Theorem \ref{esistenza 
ellittico}, while the proof of Theorem
\ref{thm par} by exploiting Theorem 
\ref{es par mu}.
\end{proof}
As a consequence of the previous results,
we obtain that the existence and persistence
of $v(x,t)$ is guaranteed if and only if
$\lambda(m)<\frac{1}{d}$, then, in order to
maximize the chances of survival we aim to minimize $\lambda(m)$.
To this goal, we need to fix a class where $m$ varies.  In particular, 
 taking also into account \cite[Theorem 3.1]{8}, an integral
bound is commonly adopted, so that we address
the Problem \eqref{problema di minimizzazione  spettrale}. 
In particular, in the following
we prove that it is solved by a bang-bang
weight.
\begin{proof}[Proof of Theorem \ref{minimo 
minimo}]
Thanks to
\eqref{condizione di crescita H}
and Proposition \ref{poincare},
it is possible to argue as in 
\cite[Theorem 3.2]{12}
to prove that $\Lambda$ is a minimum,
namely there exists $m \in \Mcal$ such that
$\Lambda=\lambda(m)$.
We now prove that the optimal weight is 
of bang-bang type.
Let $\varphi$ be the positive bounded 
eigenfunction associated with $
\Lambda$, the Bathtub principle 
(see e.g. \cite[Lemma 3.3]{12}) implies that there
exists a measurable set $\omega\subset
\Omega$ such that $\left\{ \varphi>t \right\}
\subset 
\omega \subset \left\{\varphi\geq t  
\right\}$, with $|\omega|=\frac{\beta+
m_0}{1+\beta}|\Omega|$ and 
\begin{equation}\label{peso minimo bang-bang}
\Lambda=\frac{\int_\Omega \Hcal^2(\nabla 
\varphi)}{\int_\Omega m\varphi^2}\geq 
\frac{\int_\Omega \Hcal^2(\nabla \varphi)}
{\sup_{\widetilde{m}\in \Mcal}\int_\Omega 
\widetilde{m}
\varphi^2}=\frac{\int_\Omega \Hcal^2(\nabla 
\varphi)}{\int_\Omega (\chi_\omega -
\beta\chi_{\omega^c})\varphi^2}\geq \Lambda.
\end{equation}
So that, 
$\Lambda$ is achieved by $m=\chi_\omega-\beta
\chi_{\omega^c}$.
Moreover, taking into account assumptions
\eqref{H positiva}, \eqref{H pos omogenea} and
\eqref{convessità}, and that
$\Hcal \in C^2(\mathbb{R}^N \setminus 
\left\{ 0 \right\})$, we can exploit 
\cite[Corollary 1.7]{Antonini}
to obtain that $\left| \left\{ \varphi=t
\right\}\right|=0$, concluding the proof.
\end{proof}
Knowing that the optimal set 
$\omega$ is
a superlevel set of the positive buonded
eigenfunction associated with $\Lambda$,
we want to detect it, at least in dimension
one.
With this aim,
let us take $N=1$, $\Hcal$ as in \eqref{H 
unidimensionale}, and suppose that $
\Omega=(0,1)$.
We will describe $\omega$ 
by means of
the anisotropic rearrangements studied 
in \cite{Pepisc}, see also \cite{15}.
\begin{proof}[Proof of Theorem \ref{D intervallo}]
Let us prove the result for Problem
\eqref{problema unidmensionale DN}.
Let $\varphi$ be the bounded positive 
eigenfunction associated with $
\Lambda=\lambda(\chi_\omega-
\beta\chi_{\Omega\setminus 
\omega})=\lambda(\omega)$.
Let us take $\varphi^*$ and $(m_\omega)^*$ 
respectively as the monotone increasing 
rerrangements of $\varphi$ and 
$m_\omega=\chi_\omega-
\beta\chi_{\Omega\setminus \omega}$. Then,
it results that 
$\varphi^* \in H_\Dcal^1(0,1)$ and $
(m_\omega)^* 
\in \Mcal$, indeed, $-\beta\leq 
(m_\omega)^*(x)\leq 1$ a.e. in $(0,1)$, and 
by the equimeasurability it follows that $$
\int_0^1 m_\omega^*(x)dx=\int_0^1 
m_\omega(x)dx\leq m_0.$$
Moreover, by applying the Hardy-Littlewood 
inequality (see \cite[Property P1]{
15}) it follows that 
$$\int_0^1 (m_\omega)^*(\varphi^*)^2 \geq 
\int_0^1 m_\omega \varphi^2>0,$$
so that, using Polya inequality
\cite[Proposition 4.1]{Pepisc} (see also \cite{PPS'}) one has
$$\frac{\int_0^1 \Hcal^2((\varphi^*)')}
{\int_0^1 (m_\omega)^*(\varphi^*)^2}\leq 
\frac{\int_0^1 \Hcal^2(\varphi')}{\int_0^1 
m_\omega \varphi^2}=\Lambda\leq 
\frac{\int_0^1 \Hcal^2((\varphi^*)')}
{\int_0^1 (m_\omega)^*(\varphi^*)^2},$$
yielding that $$\int_0^1 
\Hcal^2(\varphi')=\int_0^1 
\Hcal^2((\varphi^*)').$$ 
Indeed,  supposing by contradiction that $
\int_0^1 \Hcal^2((\varphi^*)')< \int_0^1 
\Hcal^2(\varphi')$, one has,
exploiting Polya inequality,
$$\frac{\int_0^1 \Hcal^2(\varphi')}
{\int_0^1 m_\omega\varphi^2}
>\frac{\int_0^1 \Hcal^2((\varphi^*)')}
{\int_0^1 m_\omega\varphi^2}\geq 
\frac{\int_0^1 \Hcal^2((\varphi^*)')}
{\int_0^1 (m_\omega)^* (\varphi^*)^2}
=\frac{\int_0^1 \Hcal^2(\varphi')}{\int_0^1 
m_\omega\varphi^2},$$ which is absurd. 
Thus, by applying \cite[Proposition 4.2]
{Pepisc} we obtain that $\varphi$ is 
monotone, and by Theorem \ref{minimo 
minimo} it follows that $\omega$ is an 
interval. Moreover, since $\varphi(0)=0$, 
and $\varphi>0$ in $(0,1)$, it is 
necessairly monotone increasing, and since 
$\omega$ is the superlevel set of $\varphi$ with 
measure $\frac{\beta+m_0}{1+\beta}$, the 
conclusion holds. If one has Problem 
\eqref{problema unidimensionale ND}, it
can be proved analogously that $\omega=\left(0,
\frac{\beta+m_0}{1+\beta}\right)$, using 
the monotone decreasing rearrangement. 
\end{proof}
\begin{figure}[h]
\begin{tikzpicture}[scale=0.025cm]
\draw[-latex](0,0) -- (5,0);
\draw[-latex](0,0) -- (0,3);
\draw(0,0) -- (-1,0);
\draw(0,0) -- (0,-1.4);
\draw[thick, samples=500, domain=0.01:1.7] plot (\x,{2 });
\draw[thick, red, samples=500, domain=0.01:1.7] plot (\x, {0});
\draw[thick,  samples=500, domain=1.7:4.7] plot (\x,{-1 });
\node at (5,-0.2) {\tiny{$x$}};
\node at (-0.2,2.85) {\tiny{$y$}};
\node at (2.2,2) {\small{$m_{\omega}$}};
\node at (1,0.3) {\textcolor{red}{$\omega$}};
\node at (-0.2,-0.3) {\tiny{$0$}};
\fill [black] (1.7,0)  circle  (0.04);
\node at (1.9,-0.3) {\tiny{$|\omega|$}};
\node at (4.7,-0.3) {\tiny{$1$}};
\node at (-0.4,-1) {\tiny{$-\beta$}};
\node at (-0.4,2) {\tiny{$1$}};
\draw[-latex](6,0) -- (11,0);
\draw[-latex](6,0) -- (6,3);
\draw(6,0) -- (5.5,0);
\draw(6,0) -- (6,-1.4);
\draw[thick,  samples=500, domain=6:9] plot (\x,{-1 });
\draw[thick, red,samples=500, domain=9:10.7] plot (\x,{0 });
\draw[thick, samples=500, domain=9:10.7] plot (\x,{2 });
\node at (11,-0.2) {\tiny{$x$}};
\node at (5.8,2.85) {\tiny{$y$}};
\node at (5.8,-0.3) {\tiny{$0$}};
\node at (10.7,-0.3) {\tiny{$1$}};
\node at (8.7,2) {\small{$m_\omega$}};
\node at (5.7,2) {\tiny{$1$}};
\node at (5.6,-1) {\tiny{$-\beta$}};
\fill [black] (9,0)  circle  (0.04);
\node at (9,-0.3) {\tiny{$1-|\omega|$}};
\node at (9.7,0.3) {\textcolor{red}{$\omega$}};
\end{tikzpicture}
\caption{Representation of the optimal weight
$m_\omega$ in dimension 1,
according to the mixed boundary conditions.}\label{figura 1}
\end{figure}
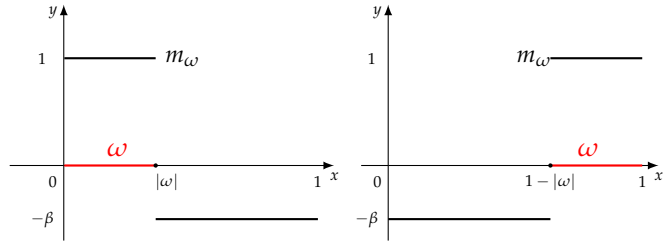
Figure \ref{figura 1} shows that, regardless of the 
anisotropy, the position of the optimal
interval depends exclusively on  the mixed 
boundary conditions are posed.
We remark that this situation is completely 
different from the one with homogeneous Neumann 
or Dirichlet
boundary conditions, since in those cases
the position of the optimal set is determined by
$a$ and $b$, as
shown in Figures \ref{figura 2} and 
\ref{figura 3} (see for instance
\cite{Pepisc}).
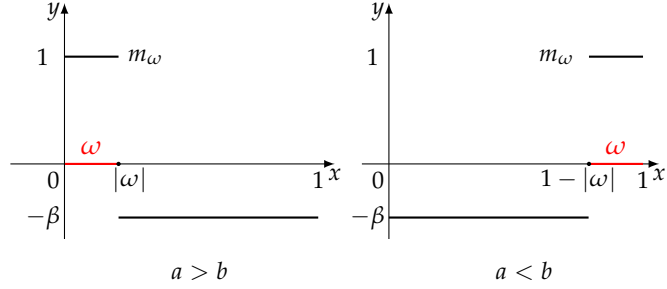
\begin{figure}[h]
\begin{tikzpicture}[scale=0.025cm]
\draw[-latex](0,0) -- (5,0);
\draw[-latex](0,0) -- (0,3);
\draw(0,0) -- (-1,0);
\draw(0,0) -- (0,-1.4);
\draw[thick, samples=500, domain=0.01:1] plot (\x,{2 });
\draw[thick, red, samples=500, domain=0.01:1] plot (\x, {0});
\draw[thick, samples=500, domain=1:4.7] plot (\x,{-1 });
\node at (5,-0.2) {\small{$x$}};
\node at (-0.2,2.85) {\small{$y$}};
\node at (1.5,2) {\small{$m_{\omega}$}};
\node at (0.5,0.3) {\textcolor{red}{$\omega$}};
\node at (-0.2,-0.3) {\small{$0$}};
\fill [black] (1,0)  circle  (0.04);
\node at (1.2,-0.3) {\small{$|\omega|$}};
\node at (4.7,-0.3) {\small{$1$}};
\node at (-0.4,-1) {\small{$-\beta$}};
\node at (-0.4,2) {\small{$1$}};
\node at (2.5,-2) {\small{$a>b$}};
\draw[-latex](6,0) -- (11,0);
\draw[-latex](6,0) -- (6,3);
\draw(6,0) -- (5.5,0);
\draw(6,0) -- (6,-1.4);
\draw[thick, samples=500, domain=6:9.7] plot (\x,{-1 });
\draw[thick, red,samples=500, domain=9.7:10.7] plot (\x,{0 });
\draw[thick, samples=500, domain=9.7:10.7] plot (\x,{2 });
\node at (11,-0.2) {\small{$x$}};
\node at (5.8,2.85) {\small{$y$}};
\node at (5.8,-0.3) {\small{$0$}};
\node at (10.7,-0.3) {\small{$1$}};
\node at (9.1,2) {\small{$m_\omega$}};
\node at (5.7,2) {\small{$1$}};
\node at (5.6,-1) {\small{$-\beta$}};
\fill [black] (9.7,0)  circle  (0.04);
\node at (9.5,-0.3) {\small{$1-|\omega|$}};
\node at (10.2,0.3) {\textcolor{red}{$\omega$}};
\node at (8.5,-2) {\small{$a<b$}};
\end{tikzpicture}
\caption{Representation of
the optimal weight $m_\omega$ in dimension 1, 
with 
homogeneous Neumann 
boundary conditions.}\label{figura 2}
\end{figure}
\begin{figure}[h]
\begin{tikzpicture}[scale=0.025cm]
\draw[-latex](0,0) -- (5,0);
\draw[-latex](0,0) -- (0,3);
\draw(0,0) -- (-1,0);
\draw(0,0) -- (0,-1.4);
\draw[thick, samples=500, domain=2.9:4.2] plot (\x,{2 });
\draw[thick, red, samples=500, domain=2.9:4.2] plot (\x, {0});
\draw[thick, samples=500, domain=0.01:2.2] plot (\x,{-1 });
\draw[thick, samples=500, domain=4.2:4.7] plot (\x,{-1 });
\node at (5,-0.2) {\small{$x$}};
\node at (-0.2,2.85) {\small{$y$}};
\node at (2.2,2) {\small{$m_{\omega}$}};
\node at (3.5,0.3) {\textcolor{red}{$\omega$}};
\node at (-0.2,-0.3) {\small{$0$}};
\fill [black] (2.9,0)  circle  (0.04);
\fill [black] (4.2,0)  circle  (0.04);
\node at (2.9,-0.3) {\small{$c_1$}};
\node at (4.2,-0.3) {\small{$c_2$}};
\node at (4.7,-0.3) {\small{$1$}};
\node at (-0.4,-1) {\small{$-\beta$}};
\node at (-0.4,2) {\small{$1$}};
\node at (2.5,-2) {\small{$a>b$}};
\draw[-latex](6,0) -- (11,0);
\draw[-latex](6,0) -- (6,3);
\draw(6,0) -- (5.5,0);
\draw(6,0) -- (6,-1.4);
\draw[thick, samples=500, domain=6:6.5] plot (\x,{-1 });
\draw[thick, samples=500, domain=7.8:10.7] plot (\x,{-1 });
\draw[thick, red,samples=500, domain=6.5:7.8] plot (\x,{0 });
\draw[thick, samples=500, domain=6.5:7.8] plot (\x,{2 });
\node at (11,-0.2) {\small{$x$}};
\node at (5.8,2.85) {\small{$y$}};
\node at (5.8,-0.3) {\small{$0$}};
\node at (10.7,-0.3) {\small{$1$}};
\node at (8.3,2) {\small{$m_\omega$}};
\node at (5.7,2) {\small{$1$}};
\node at (5.6,-1) {\small{$-\beta$}};
\fill [black] (6.5,0)  circle  (0.04);
\fill [black] (7.8,0)  circle  (0.04);
\node at (6.5,-0.3) {\small{$c_1$}};
\node at (7.8,-0.3) {\small{$c_2$}};
\node at (7.2,0.3) {\textcolor{red}{$\omega$}};
\node at (8.5,-2) {\small{$a<b$}};
\end{tikzpicture}
\caption{Representation of
the optimal weight $m_\omega$ in dimension 1,
with 
homogeneous Dirichlet 
boundary conditions. Here
$c_1=\frac{1-|\omega|}{a+b} a$ and
$c_2=\frac{|\omega|b+a}{a+b}$.}\label{figura 3}
\end{figure}
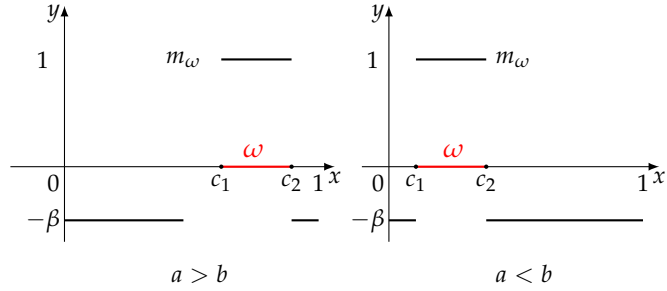
\appendix
\section{Regularity}\label{reg}
In this section we investigate the H\"older 
regularity of bounded weak solutions to 
problem
\begin{equation}\label{problema con f}
\begin{cases}
-\text{div}(\Hcal(\nabla u)\nabla_\xi 
\Hcal(\nabla u))=f \ &\text{in } \Omega,\\
\hskip105pt u>0 &\text{in } \Omega,\\
\hskip105pt u=0 &\text{on }  \Gamma_\Dcal, \\
\hskip19pt \Hcal(\nabla u)\nabla_\xi 
\Hcal(\nabla u)\cdot n =0 &\text{on } 
\Gamma_\Ncal,
\end{cases}
\end{equation}
where $f$ belongs to 
$L^r(\Omega)$, with $r>N/2$, so that the 
results apply to the positive bounded 
solutions to
Problems \eqref{problema non 
lineare}, \eqref{problema 
anisotropo} and \eqref{problema mu}.
In particular, we aim to prove the following 
result.
\begin{theorem} \label{regolarità}
Let $u \in H_\Dcal^1(\Omega)$ be a bounded weak 
solution to Problem \eqref{problema con f}. 
Then, $u \in C^{0,\gamma}
\left(\overline{\Omega}\right)
$, 
for some $0<\gamma<\frac{1}{2}$.
\end{theorem}
We will follow \cite{Stampacchia}, (see
also \cite{Colorado Peral, Colorado
Peral 2004}).
In order to do that, let us introduce some 
notations.
For $u \in H_\Dcal^1(\Omega)$ we denote 
$A(k):=\left\{ x \in \Omega \ : \ u(x)>k 
\right\}$, with $k \in \mathbb{R}$, and for 
$y \in \overline{\Omega}$ and $\rho>0$ we 
denote $\Omega(y,\rho):=B_\rho(y)\cap 
\Omega$, and 
\begin{equation}\label{Akr}
A(k;\rho):=A(k)\cap \Omega(y,
\rho).
\end{equation}
In the subsequent proof we will make use of the following Sobolev inequality  
{see \cite[p. 499]{Colorado Peral}}.
\begin{rmk}\label{rem:sob}
Since $\Omega$ is a $C^2$ 
domain,
for all $y \in \Gamma$, 
there exist $\beta>0$ and $\tilde{\rho}(y)>0$ 
such that for all $0<\rho<\tilde{\rho}(y)$ 
one has
\begin{equation}\label{sobolev+poincare}
\| u \|_{L^{2^*}(\Omega(y,\rho))}\leq \beta 
\| \nabla u \|_{L^{2}(\Omega(y,\rho))} \ \ 
\text{for all} \ u \in H_\Dcal^1(\Omega).
\end{equation}
Indeed, by the Sobolev inequality it follows 
the existence of $c>0$ such that 
\begin{equation}\label{dis sobolev}
\| u \|_{L^{2^*}(\Omega(y,\rho))}\leq c \|  
u \|_{H^1(\Omega(y,\rho))}
\end{equation}
and since $u \in H_\Dcal^1(\Omega)$, $u \in 
H_\Dcal^1(\Omega(y,\rho))$, being $
\Gamma_\Dcal(\Omega(y,\rho))=\Gamma_\Dcal 
\cap 
\Omega( y,\rho)$. Thus by applying Proposition 
\ref{poincare} we deduce that there exists $c'>0$ such that 
$$\|  u \|_{H^1(\Omega(y,\rho))}\leq c' \| 
\nabla u \|_{L^{2}(\Omega(y,\rho))},$$
so that, setting $\beta=cc'$, one has 
\eqref{sobolev+poincare}.
\end{rmk}
We observe that in every $\Omega'\subset\subset
\Omega \cup \Gamma_\Dcal$ or
$\Omega'\subset\subset
\Omega \cup \Gamma_\Ncal$ the H\"older 
regularity of solutions to Problem
\eqref{problema con f} is classical.
Hence, it 
sufficies to prove that the solutions are
H\"older continuous in $\Omega(y,\rho)$ for
every $y \in \Gamma$, where we recall that $
\Gamma=\overline{\Gamma_\Dcal}\cap \overline{
\Gamma_\Ncal}$.
Moreover, let us point out that that Theorem \ref{regolarità} can be demonstrated 
following the strategy in \cite[Section 5]{Colorado Peral 2004}.
In particular,  Theorems 5.12, 5.13 and 5.14  can be proved in the
same way, once the Theorem \ref{thm 6.2 cp} below is shown.
To state this result, let us introduce the set $K(y)$, defined, for every 
 $y \in \overline{\Omega}$, as follows
\begin{equation}\label{def K+}
K(y):=\left\{ k \in \mathbb{R} \ : \ 
\exists \ \overline{\rho}(y) \ : \, 
\forall \alpha 
\in C^1(\mathbb{R}^N), \  \ 
\text{supp}\alpha \subset B_{
\overline{\rho}(y)}(y), 
 T_k(u)\alpha \in 
H_\Dcal^1(\Omega) \ \right\},
\end{equation}
where we recall  that
$T_k(u):=\min\left\{  u,k \right\}\in H_\Dcal^1(\Omega)$ for every $
u \in H_\Dcal^1(\Omega)$.
\\
Roughly speaking, $k\in K(y)$ if there exists a sufficiently small radius $\overline{\rho}(y)$, such that $ T_k(u)\alpha(x) \in 
H_\Dcal^1(\Omega)$, for every $\alpha$ regular with support contained in $B_{
\overline{\rho}(y)}(y)$. Note that if $y\in \Gamma_{\Ncal}$, the radius $\overline{\rho}(y)$ can be large, while it is needed to be small when $y\in \overline{\Gamma_{\Dcal}}$.
\begin{theorem}\label{thm 6.2 cp}
Let $u \in H_\Dcal^1(\Omega)$ be a nonnegative weak
solution 
to Problem \eqref{problema con f}. Then, 
there exist two constants
$\Pi,\Theta>0$ depending on $\underline{\alpha
}, \overline{\alpha}, Q$,
such that for 
every $y \in \overline{\Omega}$, 
$0<\rho<R<\overline{\rho}(y)$ and every $k \in 
K(y)$ it results
\begin{equation}\label{7.1}
\int_{A(k;\rho)}|\nabla u|^2 \leq 
\frac{\Pi}{(R-\rho)^2}\int_{A(k;R)}|u-k|^2 
+\Theta
\left(\int_{A(k;R)}|f|^r\right)^{2/r}\left| 
A(k;R) \right|^{
\frac{(N+2)r-2N}{Nr}}.
\end{equation}
\end{theorem}
\begin{proof}
Let us take $G_k(u):=u-T_k(u)$, and
define the nonnegative $C^1$ function
$\alpha(r)$ as
\begin{equation}\label{def alpha}
\alpha(r)=
\begin{cases}
1 \ &\text{if } r < \rho,\\
\frac{(R-r)^2(R+2r-3\rho)}{(R-
\rho)^3} & \text{if } \rho\leq  r \leq R,\\
0 &\text{if } r>R.
\end{cases}
\end{equation}  
so that $0\leq \alpha(r)\leq 1$
and 
$\frac{[\alpha ' (r)]^2}{\alpha(r)}\leq 
\frac{36}{(R-\rho)^2}$.
In view of  \ref{def K+}, it follows that 
for every $y \in \overline{\Omega}$ and for every $k\in K(y)$, one has
$v(x):=G_k
(u)\alpha(r) \in H_\Dcal^1(\Omega)$, 
where 
$r=|x-y|$. 
Indeed,  since $k \in K(y)$, we have that
$T_k(u)\alpha(r)\in H_\Dcal^1(\Omega)$.
Taking $v$ as test 
function in the weak formulation of
\eqref{problema con f}, we get
\begin{equation}\label{test}
\int_\Omega \Hcal(\nabla u)\nabla_\xi 
\Hcal(\nabla u)\cdot \nabla v=\int_\Omega f 
v.
\end{equation} 
Since $v=\alpha(r)G_k(u)\equiv 0$ in 
$\Omega\setminus A(k;R)$,
it results
\begin{equation}\label{grad v}
\nabla v(x)=
\left\{
\alpha(r)\nabla u +\alpha'(r)(u-k)\frac{x-y}
{r}\right\} \chi_{A(k;R)},
\end{equation}
where $A(k;R)$ is defined in \eqref{Akr}. Hence,
\eqref{test} is equivalent to
\begin{align}\label{7.7}
\begin{split}
&\int_{A(k;R)}\alpha(r)\Hcal(\nabla u)
\nabla_\xi \Hcal(\nabla u)\cdot \nabla u+
\int_{A(k;R)}\alpha'(r)(u-k)\Hcal(\nabla u)
\nabla_\xi \Hcal(\nabla u)\cdot \frac{(x-y)}
{r}\\
&= \int_{A(k;R)}\alpha(r)(u-k)f.
\end{split}
\end{align}
Now, exploiting \eqref{condizione 
di crescita H} and \eqref{(2.5) Montoro} we 
obtain 
$$\underline{\alpha}^2 \int_{A(k;R)}
\alpha(r)|\nabla u|^2 \leq \int_{A(k;R)}
\alpha(r)\Hcal^2(\nabla u)=\int_{A(k;R)}
\alpha(r) \Hcal(\nabla u)\nabla_\xi 
\Hcal(\nabla u)\cdot \nabla u.$$
On the other hand, multiplying the second term
of the left-hand side by $\alpha^{1/2}(r)$,
applying 
H\"older inequality, and taking into account 
conditions \eqref{(2.5) Montoro}, 
\eqref{(2.6) Montoro}, one has
\begin{align}\label{7.8}
\begin{split}
& \left|\int_{A(k;R)}\alpha'(r)(u-k) 
\Hcal(\nabla u)\nabla_\xi \Hcal(\nabla u)
\cdot \frac{(x-y)}{r} \right|\leq \\
&\leq \left\{ \int_{A(k;R)}\alpha(r)|
\Hcal(\nabla u)\nabla_\xi \Hcal(\nabla u)|^2 
\right\}^{1/2}\left\{ \int_{A(k;R)} 
\frac{(\alpha'(r))^2}{\alpha(r)} (u-k)^2 
\right\}^{1/2}\\
& \leq \overline{\alpha}Q \Ccal(\rho,R)\left\{ 
\int_{A(k;R)}\alpha(r)|\nabla u|^2 \right\}^{1/
2}\left\{ \int_{A(k;R)} (u-k)^2 \right\}^{1/2}
\\
& \leq \frac{1}{2}\underline{\alpha}^2 
\int_{A(k;R)}\alpha(r)|\nabla u|^2 +  
\frac{\overline{\alpha}^2 Q^2\Ccal^2(\rho,R)}
{2\underline{\alpha}^2}\int_{A(k;R)}(u-k)^2,
\end{split}
\end{align}
where the last inequality follows from 
Young inequality, and 
$\mathcal{C}(\rho,R):=\left\{\max_{\rho\leq 
r\leq R}\frac{[\alpha'(r)]^2}{\alpha(r)}
\right\}^{1/2}$,
so that we have $
\Ccal^2(\rho;R)\leq \frac{36}{(R-\rho)^2}$.
Let us deal with the right-hand side in
\eqref{7.7}.
Recalling that $G_k(u)=u-k$ in 
$A(k;R)$, and that $A(k;R)\subset \Omega$,
by applying H\"older and 
Young inequalities we get
\begin{align}\label{holder}
\begin{split}
\left| \int_{A(k;R)}\alpha(r)(u-k)f \right|
&\leq
\|\alpha(r)G_k(u) \|_{L^{2^*}(\Omega(y,\rho))}\| f 
\|_{L^{\frac{2N}{N+2}}(A(k;R))}\\
&\leq 
\frac{\delta}{2}\| \alpha(r)G_k(u)
\|_{L^{2^*}(\Omega(y,\rho))}^2+\frac{1}{2\delta}\| f
\|_{L^\frac{2N}{N+2}(A(k;R))}^2 ,
\end{split}
\end{align}
for every $\delta>0$.
Moreover, in view of Remark \eqref{rem:sob} and takin into account \eqref{grad v},
we obtain
\[\begin{split}
\| \alpha(r)G_k(u) 
\|_{L^{2^*}(\Omega(y,\rho))}^2 
&
\leq \beta^2 \| \nabla (\alpha(r)G_k(u)) 
\|_{L^{2}(\Omega(y,\rho))}^2 
\\
&
\leq 2\beta^2 \left\{ \int_{A(k;R)}
[\alpha '(r)]^2 |G_k(u)|^2 +
\int_{A(k;R)} \alpha^2(r) |\nabla G_k(u)|^2 
\right\}.
\end{split}
\]
Since $[\alpha '(r)]^2 \leq \frac{36}{(
R-\rho)^2}$ and $\alpha^2(r)\leq \alpha(r)$,
we have
\begin{equation}\label{dis:1}
\| \alpha(r)G_k(u) 
\|_{L^{2^*}(\Omega(y,\rho))}^2 
\leq 
2\beta^2 \left\{ 
\frac{36}{(R-\rho)^2}\int_{A(k;R)}
|u-k|^2 +\int_{A(k;R)} \alpha(r)
|\nabla u|^2\right\}.
\end{equation}
Moreover, by applying H\"older inequality again 
it 
results
\begin{equation*}\label{holder f}
\| f 
\|_{L^{\frac{2N}{N+2}}(A(k;R))}^2=\left(\int_{A(k;R)}|f|^{\frac{2N}{N+2}}\right)^{\frac{N+2}N}
\leq
\| f \|_{L^r(A(k;R))}^2 \left| A(k;R) \right|^{
\frac{(N+2)r-2N}{Nr}}.
\end{equation*}
This and \eqref{dis:1} imply that  \eqref{holder} becomes
\begin{align}\label{7.9}
\begin{split}
\left| \int_{A(k;R)}\alpha(r)(u-k)f \right|
\leq &
\beta^2 \delta \int_{A(k;R)}\alpha(r)
|\nabla u|^2 +\beta^2 \delta 
\frac{36}{(R-\rho)^2}\int_{A(k;R)} |u-k|^2\\
&+\frac{1}{2\delta} \| f \|_{L^r(A(k;R))}^2 
\left| A(k;R) \right|^{
\frac{(N+2)r-2N}{Nr}}.
\end{split}
\end{align}
Combining \eqref{7.7}, \eqref{7.8} and 
\eqref{7.9} we obtain
\begin{align}\label{7.10}
\begin{split}
\left(
\frac{
\underline{\alpha}^2}{2} -\beta^2\delta\right)
\int_{A(k;R)}\alpha(r)|
\nabla u|^2
\leq &
\frac{36}{(R-\rho)^2}\left( \frac{
\overline{\alpha}^2 Q}{2\underline{\alpha}^2}+
\beta^2 \delta \right)
\int_{A(k;R)} |u-k|^2 \\
&+ \frac{1}
{2\delta}\| f \|_{L^r(A(k;R))}^2 
\left| A(k;R) \right|^{
\frac{(N+2)r-2N}{Nr}}.
\end{split}
\end{align} 
Finally, taking $\delta$ sufficiently small
so that $$\left(
\frac{
\underline{\alpha}^2}{2} -
\beta^2\delta\right)>0,$$
we obtain \eqref{7.1}.
\end{proof}
\section*{Acknowledgements}
Work partially supported by the INdAM-GNAMPA 
group.


\begin{thebibliography}{9}
\bibitem{Antonini} C.A. Antonini, G. Ciraolo, A. Farina, \emph{Interior regularity results for inhomogeneous anisotropic quasilinear equations}, Mathematische Annalen (2023) 387:1745–1776.
\bibitem{Alvino Ferone Trombetti} A. Alvino, V. Ferone, G. Trombetti, P. Lions, \emph{Convex symmetrization and applications}, Annales de l’I. H. P., section C, tome 14, no 2 (1997): 275-293
\bibitem{Takac}V. E. Bobkov, P. Takáč, \emph{A Strong Maximum Principle for parabolic equations
with the p-Laplacian}, J. Math. Anal. Appl. 419 (2014): 218-230.
\bibitem{Berestychi}H. Berestycki, F. Hamel, L. Roques, \emph{Analysis of the periodically fragmented environment
model: I-Species persistence}, J. Math. Biol., vol. 51, (2005): 75-113.
\bibitem{4'}J. M. Bernard, \emph{Density results in Sobolev spaces whose elements vanish on a part of the boundary}, (2011) 32:823-846.
\bibitem{6} H. Brezis, L. Oswald, \emph{Remarks on Sublinear Elliptic Equations}, non-linear Analysis Theory, Methods and Applications, vol. 10, n. 1, (1986): 55-64.
\bibitem{8}R.S. Cantrell, C. Cosner, \emph{Diffusive logistic equations with indefinite weights}, Proceedings of the Royal Society of Edinburgh, vol. 112A, (1989): 293-318.
\bibitem{10}R.S. Cantrell, C. Cosner, \emph{The effects of spatial heterogeneity in population dynamics}, J. Math. Biol. vol. 29, (1991): 315-338.
\bibitem{9}R.S. Cantrell, C. Cosner, \emph{Spatial Ecology via
Reaction-Diffusion Equations}, Wiley 2003.
\bibitem{Colorado Peral} E. Colorado, I. Peral, \emph{Semilinear elliptic problems with mixed Dirichlet–Neumann boundary conditions}, Journal of Functional Analysis, vol. 199, (2003): 468–507.
\bibitem{Colorado Peral 2004}E. Colorado, I. Peral, \emph{Eigenvalues and bifurcation for elliptic equations with  mixed Dirichlet-Neumann boundary conditions related to Caffarelli-Kohn-Nirenberg
inequalities}, Topological Methods in Nonlinear Analysis, Journal of the Juliusz Schauder Center, Volume 23, (2004): 239–273.
\bibitem{11}D.G. de Figueireido, \emph{Positive Solutions of Semilinear Elliptic Problems}, Lecture Notes in Mathematics, vol. 957, (1982): 34-87.
\bibitem{DG} F. Della Pietra, N. Gavitone, \emph{Symmetrization for Neumann Anisotropic Problems and Related Questions}, Advanced Nonlinear Studies 12 (2012), 219–235.
\bibitem{DG'} F. Della Pietra, N. Gavitone, \emph{Faber-Krahn Inequality for Anisotropic Eigenvalue Problems with Robin Boundary Conditions}, Potential Anal (2014) 41:1147–1166.
\bibitem{Denzler} J. Denzler, \emph{Bounds for the Heat Diffusion through Windows of Given Area},  Journal of Mathematical Analysis and Applications 217, 405-422 (1998).
\bibitem{Denzler'} J. Denzler, \emph{Windows Of Given Area 
with Minimal Heat Diffusion}, Transatcions of The American Mathematical Society, Volume 351, Number 2, February 1999, Pages 569-580.
\bibitem{12}A. Derlet, J.P. Gossez, P. Taká\v c, \emph{Minimization of eigenvalues for a quasilinear elliptic Neumann problem with indefinite weight}, J. Math. Anal. Appl., vol. 371, (2010): 69-79.
\bibitem{Di benedetto} E. Di Benedetto, \emph{Degenrate Parabolic Equations}, Springer, 1993.
\bibitem{Montoro Sciunzi} F. Esposito, L. Montoro, B. Sciunzi, D. Vuono, \emph{Asympthotic Behaviour of Solutions to the Anisotropic
Doubly Critical Equation}, Arxiv 2023.
\bibitem{FNO} V. Felli, B. Noris, R. Ognibene, \emph{ Eigenvalues of the Laplacian with moving mixedboundary conditions: the case of disappearing Dirichlet region}, Calc. Var. (2021) 60:12.
\bibitem{FNO'} V. Felli, B. Noris, R. Ognibene, \emph{ Eigenvalues of the Laplacian with moving mixedboundary conditions: the case of disappearing Neumann region}, preprint 2022.
\bibitem{13}R.A. Fisher, \emph{The advance of advantageous genes}, Ann. Eugenics, vol. 7, (1937): 335-369 .
\bibitem{AMMP} J. Garcia Azorero, A. Malchiodi, L. Montoro, I. Peral, \emph{ Concentration of Solutions for Some Singularly
Perturbed Mixed Problems: Existence Results}, Arch. Rational Mech. Anal. 196 (2010) 907–950.
\bibitem{AMMP'} J. Garcia Azorero, A. Malchiodi, L. Montoro, I. Peral, \emph{ Concentration of solutions for some singularly perturbed mixed problems: Asymptotics of minimal energy solutions}, Ann. I. H. Poincaré– AN 27 (2010) 37–56.
\bibitem{HKL} M. Hinterm\"uller, C. Y. Kao,
A. Laurain, \emph{ Principal Eigenvalue Minimization for an Elliptic
Problem with Indefinite Weight and Robin Boundary
Conditions}, Appl Math Optim, vol. 65, (2012):111–146.
\bibitem{Jaros} J. Jar\v os, \emph{Caccioppoli Estimates through an Anisotropic Picone's Identity}, Proceedings Of The
American Mathematical Society, Vol. 143, N. 3, March (2015) : 1137–1144.
\bibitem{15}B. Kawohl, \emph{Rearrangements and Convexity of Level Sets in PDE}, Springer-Verlag 1985.
\bibitem{Kesavan}S. Kesavan, \emph{Symmetrization and Applications}, Series in Analysis, vol. 3, World Scientific Publishing Co. Pte. Ltd., 2006.
\bibitem{17} A.N. Kolmogorov,  I.G. Petrovsky,  N.S. Piskunov,\emph{ Etude de l'equation de la diffusion avec croissance de la quantité de matière et son application à un problème biologique},
Bulletin Université d’\' Etat à Moscou (Bjul. Moskowskogo Gos. Univ.), Série internationale, vol. A 1, (1937): 1-26.
\bibitem{17'} O.A. Ladyzhenksaya, N. N. Uraltseva, \emph{Linear and Quasilinear Elliptic Equations}, Academic Press, New York, 1968.
\bibitem{18} J. Lamboley, A. Laurain, G. Nadìn, Y. Privat, \emph{Properties of optimizers of the principal eigenvalue
with indefinite weight and Robin conditions}, Calculus of Variations, vol. 55, n. 144, (2016).
\bibitem{LMPPS} T. Leonori, M. Medina, I. Peral, A.
Primo, F. Soria, \emph{Principal eigenvalue
of mixed problems with fractional Laplacian: moving
the boundary conditions}, J. Dirrefential Equations,
v. 265, (2018): 593-619.
\bibitem{Lions} J. L. Lions, \emph{Quelques methodes de resolution des problemes aux limites non lineaires}, Dunod, 1969.
\bibitem{20}H. Lou, \emph{On singular sets of local solutions to p-Laplace equations}, Chin. Ann. Math., vol. 29B, n. 5, (2008): 521-530.
\bibitem{21}Y. Lou, E.Yanagida, \emph{Minimization of the Principal Eigenvalue for an Elliptic Boundary Value Problem with Indefinite Weight, and Applications to Population Dynamics}, Japan J. Indust. Appl. Math., vol. 23, (2006): 275-292.
\bibitem{Pepisc}B. Pellacci, G. Pisante, D. Schiera, \emph{Spectral Optimization for Weighted Anisotropic Problems with Robin Conditions}, Journal of Differential Equations, (2024) : 303-338.
\bibitem{PPS'}B. Pellacci, G. Pisante, D. Schiera, \emph{Corrigendum to "Spectral Optimization for Weighted Anisotropic Problems with Robin Conditions"}, 2025, https://arxiv.org/pdf/2503.05294.
\bibitem{Salsa Verzini} S. Salsa, G. Verzini \emph{
Partial Differential Equations in Action From Modelling to Theory}, IV edition, Springer, 2022.
\bibitem{Sattinger} D.H. 
Sattinger, \emph{Monotone Methods in Nonlinear Elliptic and Parabolic Boundary Value Problems}, Indiana University Mathematics Department, 1972, Vol. 21, No. 11, pp. 979-1000.
\bibitem{Shamir}  E. Shamir, \emph{Regularization of second-order elliptic problems}, Israel J. Math. 6 (1968): 150–168.
\bibitem{Showalter}R. E. Showalter, \emph{Monotone Operators in Banach Space and Nonlinear Partial Differential Equations}, American Mathematical Society, 1997.
\bibitem{Simon} J. Simon,
\emph{Compact Sets in the Space 
$L^p(0,T;B)$}, Annali di Matematica Pura ed
Applicata, v.146, (1986): 65-96.
\bibitem{23} J.G. Skellam, \emph{Random dispersal in theoretical populations}, Biometrika, vol. 38, (1951): 196-218.
\bibitem{Stampacchia} G. Stampacchia, \emph{Problemi al contorno ellitici, con dati discontinui, dotati di soluzionie hölderiane}, Ann. Mat. Pura Appl. n.4, vol. 51 (1960): 1–37
\bibitem{24}M. Struwe, \emph{Variational Methods}, Springer 2008.
\bibitem{Tolksdorf} P. Tolksdorf, \emph{Regularity for a More General Class of Quasilinear Elliptic Equations 
}, Journal Of Differential Equations, vol. 51, (1984) : 126-150.
\bibitem{Trudinger} N.S. Trudinger, \emph{On Harnack Type Inequalities and Their Application to Quasilinear Elliptic Equations}, Communications On Pure And Applied Mathematics, vol. 20, (1967) : 721-747.
\end{thebibliography}
\end{document}